\documentclass{amsart}
\usepackage{graphicx,amsmath,amsfonts,latexsym,amsthm,amssymb,mathrsfs,mathtools} %
\usepackage[colorlinks=true, citecolor=red]{hyperref}

\usepackage{comment}
\usepackage{enumitem}
\usepackage{microtype}
\usepackage{xcolor}

\usepackage{mathabx}
\usepackage{microtype}
\usepackage[parfill]{parskip}

\usepackage{geometry}
\usepackage{zref-clever} 
\zcsetup{cap = true, nameinlink = false}

\usepackage{xparse} 

\newtheorem{thm}{Theorem}[section]
\zcRefTypeSetup{thm}{Name-sg = Theorem}

\NewCommandCopy{\newtheoremcopy}{\newtheorem} 
\RenewDocumentCommand{\newtheorem}{m O{thm} m}{ 
\newtheoremcopy{#1}[#2]{#3}
\AddToHook{env/#1/begin}{\zcsetup{countertype={thm=#1}}}
\zcRefTypeSetup{#1}{Name-sg = #3}
}

\newtheorem{lem}[thm]{Lemma}
\newtheorem{prop}[thm]{Proposition}
\newtheorem{cor}[thm]{Corollary}
\newtheorem{conj}[thm]{Conjecture}
\newtheorem{quest}[thm]{Question}

\theoremstyle{definition}
\newtheorem{defn}[thm]{Definition}
\newtheorem{eg}[thm]{Example}
\newtheorem{noneg}[thm]{Non-example}
\newtheorem{rmk}[thm]{Remark}
\newtheorem{construction}[thm]{Construction}

\AtBeginEnvironment{defn}{\pushQED{\hfill$\lozenge$}}
\AtEndEnvironment{defn}{\popQED}
\AtBeginEnvironment{eg}{\pushQED{\hfill$\lozenge$}}
\AtEndEnvironment{eg}{\popQED}
\AtBeginEnvironment{noneg}{\pushQED{\hfill$\lozenge$}}
\AtEndEnvironment{noneg}{\popQED}
\AtBeginEnvironment{rmk}{\pushQED{\hfill$\lozenge$}}
\AtEndEnvironment{rmk}{\popQED}

\newcommand{\cref}[1]{\zcref{#1}}
\newcommand{\Cref}[1]{\zcref{#1}}

\renewcommand{\phi}{\varphi}

\newcommand{\cone}{\mathrm{cone}}

\newcommand{\cO}{\mathcal O}
\newcommand{\cP}{\mathcal P}

\DeclareMathOperator{\sing}{\mathrm{sing}}
\DeclareMathOperator{\intr}{\mathrm{int}}
\DeclareMathOperator{\lk}{\mathrm{lk}}

\usepackage{todonotes} %
\title{Finiteness of veering triangulations}

\author[Thomas Barthelm\'e]{Thomas Barthelm\'e}
\address{Queen's University, Kingston, Ontario}
\email{thomas.barthelme@queensu.ca}
\urladdr{sites.google.com/site/thomasbarthelme}

\author[Chi Cheuk Tsang]{Chi Cheuk Tsang}
\address{Tongji University, Shanghai 200092, China}
\email{chicheuk@tongji.edu.cn}
\urladdr{https://chicheuktsang.github.io}

\author[Jonathan Zung]{Jonathan Zung}
\address{Georgia Institute of Technology\\
Atlanta, GA\\
USA
}
\email{jzung3@gatech.edu}

\begin{document}

\begin{abstract}
We show that a finite volume cusped hyperbolic 3-manifold admits at most finitely many veering triangulations, and in fact this number is bounded above by the number of Giroux torsion free tight contact structures. This resolves Kirby problem K3 3.21e.  More generally, for any 3-manifold $M$ and any link $L$, we show finiteness for the family of pseudo-Anosov flows which admits a \emph{strictly positive Birkhoff section relative to $L$}. Combined with work of Li, this implies that a fixed closed $3$-manifold admits at most finitely many pseudo-Anosov flows without perfect fits. This resolves Kirby problem K3 3.21d.
\end{abstract}

\maketitle

\section{Introduction}

The problem of classifying (pseudo-)Anosov flows on $3$-manifolds up to orbit-equivalence, as introduced by Smale \cite{Sma67}, can be split into three main steps: Determining which $3$-manifolds admit pseudo-Anosov flow (the \emph{existence} question), determining how many distinct pseudo-Anosov flows a given $3$-manifold admits (the \emph{abundance} question), and finally giving algebraic/numerical complete invariants to classify pseudo-Anosov flows on a given manifold (the \emph{classification} question). 

In this article, we are concerned with the abundance question. There are now many known examples of $3$-manifolds admitting \emph{arbritrarily many} Anosov flows \cite{BBY17, BM22, ClaPin25, BY24, BSZ25}. Here we consider the last main open part of that question, which is the so-called \emph{Finiteness Conjecture}:
\begin{conj}
    Let $M$ be a fixed closed $3$-manifold. Then $M$ admits at most finitely many pseudo-Anosov flows up to orbit equivalence.
\end{conj}

This conjecture, now stated as Problem 2.31(b) of Kirby's K3 list \cite{K3_list}, has seen a lot of development in the past few years thanks in large part to the incorporation of techniques coming from contact geometry, together with the complete invariant of transitive pseudo-Anosov flows introduced in \cite{BBowM24,BFraM25}. More precisely, in \cite{BBowM24}, Barthelmé--Bowden--Mann proved the finiteness conjecture for the class of Anosov Reeb flows. This was later extended by Zung \cite{Zun25} and Baldwin--Sivek--Zung \cite{BSZ25} to pseudo-Anosov flows admitting positive Birkhoff sections\footnote{A result of Marty \cite{Mar25} gives that an Anosov flow is orbit equivalent to an Anosov Reeb flow if and only if it admits a positive Birkhoff section.} in rational homology spheres, and finally by Chaidez--Pan \cite{ChaPan25} to a class that they introduce and call \emph{pseudo-Anosov Reeb flows}. %

There are further works in progress regarding this conjecture of which we are aware, and which do not use contact geometry: Landry--Taylor \cite{LT26} prove that every finite-depth foliation in a hyperbolic $3$-manifold admit finitely many almost transverse pseudo-Anosov flows, while Barthelmé--Paulet \cite{BP26} prove finiteness of transitive pseudo-Anosov flows in graph-manifolds up to finite covers.

Given the correspondence, introduced by Agol and Guéritaud, between pseudo-Anosov flows and \emph{veering triangulations},  the finiteness conjecture for pseudo-Anosov flows admits a veering triangulation counterpart (Problem 3.21(e) in \cite{K3_list}), which we prove here:
\begin{thm}\label{thmintro_veering}
    Let $M^\circ$ be an orientable compact 3-manifold with torus boundary components. Then there are at most finitely many veering triangulations on $M^\circ$ up to isotopy.
\end{thm}

The above theorem is in fact a consequence of our main result, which expands on the strategy of the works \cite{BBowM24,Zun25,BSZ25,ChaPan25}, by proving finiteness of pseudo-Anosov flows that admit a \emph{strictly positive Birkhoff section relative to a link.} 
Precisely, given a link $L$ in a manifold $M$, we say that a pseudo-Anosov flow $\phi$ admits a \emph{positive Birkhoff section relative to $L$}, if $\phi$ admits a Birkhoff section $S$ such that all the \emph{negative} boundary components of $S$ are in $L$. We say that it is \emph{strictly} positive if additionally, any surface that is almost transverse to $\phi$ must intersect $L$ transversely. (See Section \ref{subsec_birkhoff} for details.)

With this terminology, the key result proven here is
\begin{thm}\label{thm_main}
    Let $M$ be an oriented closed 3-manifold and let $L$ be a link in $M$.
Then there are only finitely many distinct pseudo-Anosov flows on $M$ that are strictly positive relative to $L$, up to orbit equivalence.
\end{thm}

The finiteness of veering triangulations follows from Theorem \ref{thm_main}, thanks to a result of Tsang \cite{Tsa24a}, implying that the pseudo-Anosov flows constructed from veering triangulation will have strictly positive Birkhoff section relative to the link corresponding to $\partial M^\circ$.

Furthermore, a result of Li \cite{Li26} allows us to deduce another finiteness result from Theorem \ref{thm_main}:
In \cite{Li26}, Li proved that, given an atoroidal manifold $M$, there are only finitely many distinct isotopy classes of possible singular orbits of pseudo-Anosov flows. Thus we get:

\begin{cor}
    Let $M$ be a closed oriented atoroidal manifold. There are only finitely many distinct pseudo-Anosov flows with a positive Birkhoff section relative to its singularities on $M$, up to isotopy equivalence\footnote{We say that two flows are \emph{isotopically equivalent} if they are orbit equivalent via an homeomorphism that is isotopic to identity.}.    
\end{cor}

Finally, using the construction of Birkhoff sections in \cite{Tsa24a} once again, we deduce 
\begin{cor}\label{cor_finiteness_no_perfect_fits}
    Let $M$ be a closed manifold. There are only finitely many pseudo-Anosov flows \emph{without perfect fits} on $M$.
\end{cor}
This corollary answers Problem 3.21(d) of \cite{K3_list}.

Theorem \ref{thm_main} formally generalizes the finiteness results obtained in \cite{Zun25} and \cite{BSZ25}, and by Asaoka--Bonatti--Marty \cite{ABM24}, also the finiteness of Anosov Reeb flows obtained in \cite{BBowM24}. Indeed, all these previous results can be seen as cases of Theorem \ref{thm_main} where the link $L$ is empty.

The strategy to prove Theorem \ref{thm_main} is similar to those of \cite{BBowM24,Zun25,BSZ25,ChaPan25}: One wants to translate the problem of distinguishing pseudo-Anosov flows up to isotopy equivalence into a problem of distinguishing contact structures of Reeb flows up to isotopy. Morally, the key is to show that one can replace a pseudo-Anosov flow by a Reeb flow with the same \emph{spectrum}, or \emph{free homotopy data}. The spectrum is a complete invariant of (most) transitive pseudo-Anosov flows by \cite{BFraM25}, and on the Reeb side, cylindrical contact homology gives that it is an invariant of contact structures up to isotopy. Finiteness then follows from Colin--Giroux--Honda's finiteness theorem for tight contact structures \cite{CGH09}.

This strategy works directly in the case of Anosov Reeb flows, but since a typical pseudo-Anosov flow is not the Reeb flow of a contact structure, going from pseudo-Anosov to Reeb dynamics becomes one of the main sticking points in the works extending \cite{BBowM24}. The new observation in our work is that, after drilling out the periodic orbits on the link $L$, it is possible to realize a pseudo-Anosov flow that is strictly positive relative to $L$ as (a modification of) the Reeb flow of a tight contact structure on the drilled manifold.

\begin{rmk}[Relation to upcoming works of Gabai--Li and Bowden--Colin] \label{rmk:gabaili}
The results of this paper were announced in January 2026,
with a preprint being circulated at the time. In the Summer of 2026, Dave Gabai and Lily Li announced that they are able to prove the Finiteness Conjecture on closed \emph{atoroidal} 3-manifolds. Concurrently, Jonathan Bowden and Vincent Colin obtained another proof of the Finiteness Conjecture for Anosov, and some pseudo-Anosov, flows on atoroidal 3-manifolds in a work in preparation, a preprint of which started being circulated in August 2026.

These upcoming results both generalize Corollary \ref{cor_finiteness_no_perfect_fits}, and recover Theorem \ref{thmintro_veering}.

From our understanding, Gabai and Li use completely different techniques than ours: 
They put the stable and unstable foliations of flows in Kneser normal form with respect to a fixed triangulation, then use a compactness argument to reach a contradiction.
Meanwhile, we use the classification of flows from their spectrum, as well as cylindrical contact homology. 
Our methods also give a somewhat explicit upper bound in terms of contactomorphism classes of tight contact structures, at least for the results that do not rely on Li's finiteness theorem for singular orbits.

 The proof of Bowden and Colin on the other hand does rely on contact geometry and the free homotopy data, but not on cylindrical contact homology, as they manage to show that the \emph{bicontact} structures associated to Anosov flows allow one to recover the free homotopy data in the atoroidal case. In particular, their bound is also somewhat explicit, and depends on the number of (isotopy classes of) contact structures on $M$ and the number of elements in the mapping class group of $M$.
\end{rmk}

\subsection*{Acknowledgement}
The authors thank Julian Chaidez, Yijie Pan and Lily Li.
This project originated over discussions at the `Low dimensional topology and Floer theory' workshop at Centre de recherches mathématiques in Summer 2025. 
We would like to thank the organizers for creating a stimulating research environment.
TB was partially supported by NSERC Discovery (RGPIN-2024-04412) and Alliance International programs (ALLRP 598447 - 24).
CCT was supported by a CRM postdoctoral fellowship at CIRGET. JZ was supported by a postdoctoral fellowship under Simons Foundation Award \#994330, \textit{Simons Collaboration on New Structures in Low-Dimensional Topology}.

\section{Pseudo-Anosov flow}

\subsection{Preliminaries}

Throughout this paper, unless otherwise stated, $M$ will denote an oriented closed 3-manifold.

A \textbf{pseudo-Anosov flow} on $M$ is a continuous flow $\phi:M \times \mathbb{R} \to M$ for which there exists a pair of transverse singular 2-dimensional foliations $(\mathcal{F}^s,\mathcal{F}^u)$ such that
\begin{itemize}
    \item $\mathcal{F}^s$-leaves and $\mathcal{F}^u$-leaves intersect in flow lines,
    \item flow lines along each $\mathcal{F}^s$-leaf converge in forward time and diverge in backward time, and
    \item flow lines along each $\mathcal{F}^u$-leaf diverge in forward time and converge in backward time.
\end{itemize}
For more details, we refer the reader to \cite{BM25}.

The singularity locus of $\mathcal{F}^s$ and $\mathcal{F}^u$ must coincide and equal to a collection of closed orbits.
We call these closed orbits the \textbf{singular orbits} of $\phi$, and denote it by $\sing(\phi)$.
See \Cref{fig:paflow} for an illustration of the dynamics of $\phi$ near a nonsingular and a singular orbit.

\begin{figure}
    \centering
    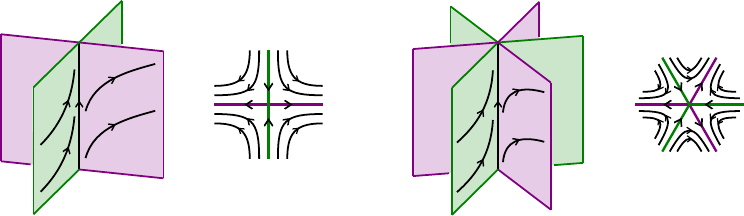
    \caption{Local pictures of a pseudo-Anosov flow near a nonsingular orbit (left) and singular orbit (right).}
    \label{fig:paflow}
\end{figure}

Two flows $\phi_1$ and $\phi_2$ on $M$ are \textbf{orbit equivalent} if there exists a homeomorphism $h\colon M \to M$ sending the flow lines of $\phi_1$ to the flow lines of $\phi_2$ in a way that preserves orientation but not necessarily preserving parametrization.
Furthermore, $\phi_1$ and $\phi_2$ are \textbf{isotopically equivalent} if $h$ can be chosen to be isotopic to identity.

The dynamics of a pseudo-Anosov flow is very rigid; in most cases, it is completely determined by its closed orbits.
To make this precise, we state the following definitions.

\begin{defn}
Let $\phi$ be a flow on an oriented closed 3-manifold $M$. The \textbf{spectrum} of $\phi$ is the set
$$\mathcal{P}(\phi) = \{g \mid \text{ $g$ is the free homotopy class of a closed orbit $\gamma$ of $\phi$}\}.$$
\end{defn}

\begin{defn}
Let $\phi$ be a flow on an oriented closed 3-manifold $M$. We say that $\phi$ is \textbf{transitive} if it has a dense orbit.
It can be shown that a pseudo-Anosov flow is non-transitive if and only if it has a separating transverse torus (\cite{Mos92a,BBonM24}).
\end{defn}

\begin{defn}
Let $\phi$ be a pseudo-Anosov flow on an oriented closed 3-manifold $M$, and let $P$ be a Seifert piece of the JSJ decomposition of $M$. %
We say that $P$ is a \textbf{scalloped periodic Seifert piece} for $\phi$ if $P$ is a Seifert piece such that all the boundary surfaces in $\partial P$ are transverse to $\phi$ and such that, for each $S\subset \partial P$, two linearly independent elements of $\pi_1(S)$ are represented by periodic orbits of $\phi$.
\end{defn}

\begin{thm}[{\cite{BFraM25}}] \label{thm:specdeterminespaflow}

Let $\phi$ be a transitive pseudo-Anosov flow with $s$ scalloped periodic Seifert pieces. Then $\mathcal{P}(\phi)$ determines the isotopy equivalence class of $\phi$ up to $2^s$ possibilities. 
\end{thm}

As the notion of scalloped periodic Seifert piece is only introduced to state the above result optimally, we do not go into details (see \cite{BFenM25} for more). The takeaway for the reader is that having a scalloped periodic Seifert piece is a fairly restrictive condition, and moreover, for each such region there is only an additional choice of \emph{sign} that is needed to get uniqueness of the isotopy class (see \cite{BFraM25}).

\begin{rmk}
In \cite{BBowM24,BFraM25} the spectrum, or free homotopy data, of a pseudo-Anosov flow is defined as the elements of $\pi_1(M)$ that represent unoriented orbits instead of oriented orbits as here. We make this different choice here to make it more consistent with the correspondence with cylindrical contact homology. The main theorem of \cite{BFraM25} works for either definition of $\cP(\phi)$, up to taking the correct notion of orbit equivalence (i.e., here the orbit equivalences we consider preserves the direction of the flow, while they may preserve or flip them in \cite{BFraM25}).
\end{rmk}

One of the key ideas in this paper is that $\mathcal{P}(\phi)$ can be split into a `peripheral' part and a `non-peripheral' part.
Here, peripherality is taken with respect to a 3-manifold with boundary $M^\circ$ that is the complement of a tubular neighborhood of a collection $\mathcal{C}$ of closed orbits.

The local intersections between the half-leaves of $\mathcal{F}^s$ containing $\mathcal{C}$ and $\partial M^\circ$ determines an isotopy class of multicurve on each component of $\partial M^\circ$. We refer to the isotopy classes of these multicurves as the \textbf{degeneracy curves}, and the slopes of these multicurves as the \textbf{degeneracy slopes}.
These constitute the `peripheral' part of the spectrum. 
The `non-peripheral' part will be discussed in the next subsection.

We also recall the following property, which will allow us to reduce to the case when the underlying 3-manifold is orientable when showing finiteness.

\begin{prop}\label{prop_finite_cover}
If $\phi,\psi$ are two transitive pseudo-Anosov flows on a manifold $N$ and such that there exists a finite cover $\hat N$ of $N$ on which the lifted flows $\widehat \phi$ and $\widehat\psi$ are isotopically equivalent, then $\phi,\psi$ are isotopically equivalent on $N$.
\end{prop}
\begin{proof}
Up to taking a further finite cover, we may assume that $\widehat N$ is a regular cover of $N$, i.e. $\pi_1(\widehat N)$ is a normal subgroup of $\pi_1(N)$. 

Since $\widehat \phi$ and $\widehat\psi$ are isotopically equivalent, the actions of $\pi_1(\widehat N)$ on the orbit spaces of $\phi$ and $\psi$ are conjugated by an homeomorphism $h\colon \cO_\phi \to \cO_\psi$, where $\cO_\phi, \cO_\psi$ are the orbit spaces of $\phi$ and $\psi$ respectively. Now consider any $g\in \pi_1(M)$ and $x\in \cO_\phi$ a non-corner fixed point, i.e. $x$ is the unique fixed point of some element $f\in \pi_1(\widehat N)$. Then $g\cdot x$ is the unique fixed point of $gfg^{-1} \in \pi_1(\widehat N)$ in $\cO_\phi$, and therefore $h(g x)$ is the unique fixed point of $gfg^{-1}$ in $\cO_\psi$, so corresponds to $g\cdot h(x)$. By density of non corner fixed points for transitive flows (\cite{BFraM25}) and continuity of the action, we deduce that for any $y\in \cO_\phi$, $h(g\cdot y) = g\cdot h(y)$ for all $g\in \pi_1(N)$. So $\phi$ and $\psi$ are isotopically equivalent.
\end{proof}

\subsection{Blow up}
In \cite{Zun25}, a special type of blow-up construction was introduced which allows one to modify a pseudo-Anosov flow near a periodic orbit without losing control of the free homotopy classes of periodic orbits. This construction will be essential for us, and we recall it here:

\begin{construction}[Solid blow-up of a flow]\label{const:blowup}
    Let $\gamma$ be a closed orbit of a pseudo-Anosov flow $\phi$ on an oriented closed 3-manifold $M$. We can modify $\phi$ near $\gamma$ into a flow $\phi_{\bullet \gamma}$ on $M$, so that there is a closed tubular neighborhood $\nu$ of $\gamma$ with the following properties:
\begin{itemize}
    \item The restriction of $\phi_{\bullet \gamma}$ to $M \backslash \nu$ is isotopically equivalent to the restriction of $\phi$ to $M$.
    \item The restriction of $\phi_{\bullet \gamma}$ to $\partial \nu$ consists of a finite number of closed orbits. Each such closed orbit is a positive hyperbolic closed orbit of $\phi_{\bullet \gamma}$. The isotopy class of these closed orbits coincides with the local intersection between the half-leaves of $\mathcal{F}^s$ containing $\gamma$ and $\partial \nu$.
    \item The restriction of $\phi_{\bullet \gamma}$ to $\intr \nu$ is isotopically equivalent to a flow of the form 
    \begin{equation} \label{eq:blowupreebflow}
    (r,\theta,z) \mapsto (r,\theta+s(r)t,z+t)
    \end{equation}
    in cylindrical coordinates, where $r\in [0,1]$, $\theta \in \mathbb R/\mathbb Z$, $t\in \mathbb R/\mathbb Z$, for some function $s$ with the following properties:
    \begin{itemize}
        \item $s$ extends to a smooth even function around $0$.
        \item $s$ is decreasing near $\partial \nu$ and takes values in $(-1,0)$ on $r > 0$. 
        
        \item Let $s_1 = \lim_{r \to 1^-} s(r)$ be the limiting value of $s$ on $\partial \nu$. Then the degeneracy slope is $s_1 \frac{\partial}{\partial \theta} + \frac{\partial}{\partial z}$.
    \end{itemize}
    In particular, the fact that $s$ takes non-integral values on $r > 0$ implies that the core of $\nu$ is an elliptic orbit of $\phi_{\bullet \gamma}$, and it is the unique closed orbit in $\intr \nu$ in its homotopy class.
\end{itemize}

See \Cref{fig:paflowblowup} and \cite[Section 3.1]{Zun25}, where this operation is referred to as \textit{counterclockwise blow-up}. 
We refer to $\phi_{\bullet \gamma}$ as a \textbf{solid blow-up} of $\phi$ at $\gamma$, and refer to $\nu$ as the \textbf{blow-up region}.
Generalizing this operation, we can talk about a solid blow-up $\phi_{\bullet \mathcal{C}}$ of $\phi$ at any collection $\mathcal{C}$ of closed orbits.
\end{construction}

\begin{rmk}
    When we do a solid blow-up $\phi_{\bullet \mathcal{C}}$ on a collection of orbits, we have as many functions $s$ and local cylindrical coordinates as there are components of $\mathcal C$. However, to spare everyone the use of more indices, we will slightly abuse notations, and use the same notation $s(r)$ for each instance of the functions $s$.
\end{rmk}

\begin{figure}
    \centering
    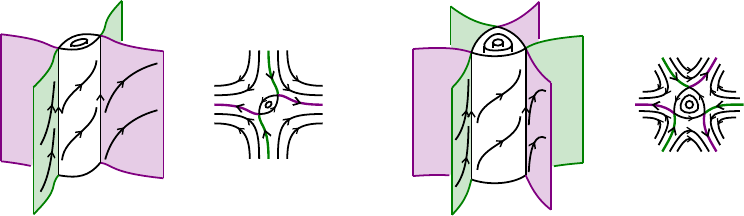
    \caption{Local picture of solid blowing up a pseudo-Anosov at a nonsingular orbit (left) and a singular orbit (right). Compare with \Cref{fig:paflow}.}
    \label{fig:paflowblowup}
\end{figure}

The word `solid' in `solid blow-up' is to distinguish it from the following variation of the construction: 
\begin{defn}[Boundary blow-up]
Let $\gamma$ be a closed orbit of a pseudo-Anosov flow $\phi$ on an oriented closed 3-manifold $M$.
A \textbf{boundary blow-up} of $\phi$ at $\gamma$, denoted by $\phi^{\circ \gamma}$, is defined by doing a solid blow-up construction on $\phi$ at $\gamma$ then restricting the solid blown-up flow to the complement of an open flow-invariant tubular neighborhood $\nu_0 = \{r < r_0\}$ of $\gamma$, where $r_0 < 1$ but is close enough to $1$ so that the function $s$ from \Cref{eq:blowupreebflow} is decreasing on $[r_0,1]$.
In particular $\phi^{\circ \gamma}$ is a flow defined on the manifold with boundary $M^{\circ \gamma} \cong M \backslash \nu_0$.

More generally, given a collection $\mathcal{C}$ of closed orbits of $\phi$, we can consider the boundary blow-up $\phi_{\circ \mathcal{C}}$ of $\phi$ at $\mathcal{C}$. Note that while a flow and its solid blow-up are on the same manifold, a flow and its boundary blow-up live on two formally distinct manifolds.
\end{defn}

Yet more generally, given two disjoint collections $\mathcal{C}_0, \mathcal{C}_1$ of closed orbits of $\phi$, we can boundary blow-up at $\mathcal{C}_0$ and solid blow-up at $\mathcal{C}_1$ to get a flow $\phi^{\circ \mathcal{C}_0}_{\bullet \mathcal{C}_1}$ on $M^{\circ \mathcal{C}_0}$.
This is done by first solid blowing-up at $\mathcal{C}_0 \cup \mathcal{C}_1$ then removing flow-invariant tubular neighborhoods of $\mathcal{C}_0$.
When $\mathcal{C}_0$ or $\mathcal{C}_1$ are understood, we will just write $\phi^\circ$ for $\phi^{\circ \mathcal{C}_0}$ and $\phi_\bullet$ for $\phi_{\bullet \mathcal{C}_1}$ etc.

Continuing the discussion on the spectrum $\mathcal{P}(\phi)$ from the last subsection, the following definition captures its `non-peripheral' part.
\begin{defn}
Let $\psi$ be a flow on an oriented 3-manifold with boundary $M^\circ$. The \textbf{primitive, non-peripheral spectrum} of $\psi$ is the set
\begin{align*}
\mathcal{P}^\circ(\psi) = \{g \mid & \text{ $g$ is the free homotopy class of a closed orbit $\gamma$ of $\psi$},\\
& \text{ where $g$ is primitive and non-peripheral}\}.
\end{align*}
\end{defn}

Note that although the isotopy equivalence class of a boundary blow-up $\phi^{\circ}$ is not uniquely determined by $\phi$, the set $\mathcal{P}^\circ(\phi^{\circ})$ is uniquely determined by $\phi$.

\begin{prop} \label{prop:nonperispecdeterminesspec}
Let $\phi$ be a pseudo-Anosov flow and let $\mathcal{C}$ be a collection of closed orbits of $\phi$. Then $\mathcal{C}$ and $\mathcal{P}^\circ(\phi^{\circ \mathcal{C}})$ uniquely determine $\mathcal{P}(\phi)$. 
\end{prop}
\begin{proof} 
Let $[\mathcal{C}]_+$ be the set of positive multiples of free homotopy classes of orbits in $\mathcal{C}$.
Let $\mathcal{P}^\circ(\phi^{\circ})_+$ be the set of positive multiples of elements of $\mathcal{P}^\circ(\phi^{\circ})$, considered as free homotopy classes in $M$.
We claim that $\mathcal{P}(\phi) = [\mathcal{C}]_+ \cup \mathcal{P}^\circ(\phi^{\circ})_+$.

The backward inclusion $\mathcal{P}(\phi) \supset [\mathcal{C}]_+ \cup \mathcal{P}^\circ(\phi^{\circ})_+$ is clear.
For the forward inclusion, let $\gamma$ be a closed orbit of $\phi$. Suppose the free homotopy class $[\gamma]$ is a positive multiple of a primitive element $g$.
By \cite[Corollary 1.4.2]{BM25}, there exists a closed orbit with free homotopy class $g$. Up to replacing $\gamma$ with this orbit, we can assume that $[\gamma]$ is primitive.
If $\gamma$ is homotopic to an orbit in $\mathcal{C}$, we are done.
Otherwise, $\gamma$ is a closed orbit of $\phi^{\circ}$ with primitive non-peripheral free homotopy class, so $[\gamma] \in \mathcal{P}^\circ(\phi^{\circ})_+$.
\end{proof}

When combined with \Cref{thm:specdeterminespaflow}, this gives the following corollary.

\begin{cor} \label{cor:nonperispecdeterminespaflow}

Let $\phi$ be a transitive pseudo-Anosov flow with $s$ scalloped periodic Seifert pieces. Then $\mathcal{C}$ and $\mathcal{P}^\circ(\phi^{\circ \mathcal{C}})$ determine the isotopy equivalence class of $\phi$ up to $2^s$ possibilities.
\end{cor}

We remark that the proof for \Cref{prop:nonperispecdeterminesspec} can be simplified if we omit the word `primitive' from the definition of $\mathcal{P}^\circ$. 
However, we choose to include it in order to simplify the following lemma regarding Lefschetz indices.

\begin{lem} \label{lem:blowupnonperispec}
Let $\phi$ be a pseudo-Anosov flow, and let $\mathcal{C}_0, \mathcal{C}_1$ be two disjoint collections of closed orbits of $\phi$. 
Then:
\begin{enumerate}
    \item Each closed orbit of $\phi^{\circ \mathcal{C}_0}_{\bullet \mathcal{C}_1}$ is homotopically essential in $M^{\circ \mathcal{C}_0}$. In fact, each closed orbit of $\phi^{\circ \mathcal{C}_0}_{\bullet \mathcal{C}_1}$ is freely homotopic to a closed orbit of $\phi$ in $M$, hence also homotopically essential in $M$.
    \item $\mathcal{P}^\circ(\phi^{\circ \mathcal{C}_0}_{\bullet \mathcal{C}_1}) = \mathcal{P}^\circ(\phi^{\circ \mathcal{C}_0})$.
    \item For every $g \in \mathcal{P}^\circ(\phi^{\circ \mathcal{C}_0}_{\bullet \mathcal{C}_1})$, the sum of Lefschetz indices of closed orbits of $\phi^{\circ \mathcal{C}_0}_{\bullet \mathcal{C}_1}$ with free homotopy class $g$ is well-defined (possibly infinite) and nonzero.
\end{enumerate}
\end{lem}
\begin{proof}

The closed orbits of $\phi^\circ$ agree with that of $\phi$ outside of the blow-up region around $\mathcal{C}_0$. 
Meanwhile, inside the blow-up region around $\mathcal{C}_0$, each closed orbit has peripheral free homotopy class, thus do not count towards $\mathcal{P}^\circ(\phi^\circ)$. 
Moreover, such peripheral closed orbits are freely homotopic to the corresponding element of $\mathcal{C}_0$.
Finally, for every $g \in \mathcal{P}^\circ(\phi^\circ)$, one of the following holds:
\begin{enumerate}
    \item All closed orbits of $\phi^\circ$ with free homotopy class $g$ are positive hyperbolic, thus the sum of Lefschetz indices is a well-defined number between $-1$ and $-\infty$.
    \item There is exactly one closed orbit of $\phi^\circ$ with free homotopy class $g$ and it is negative hyperbolic, thus the sum of Lefschetz indices is $+1$.
\end{enumerate}

Now going from $\phi^\circ$ to $\phi^\circ_\bullet$, all the closed orbits of $\phi^\circ_\bullet$ in the blow-up region around $\mathcal{C}_1$, except the core orbits and possibly the closed orbits on the boundary of the blow-up region, do not count towards $\mathcal{P}^\circ(\phi^\circ_\bullet)$ since they have non-primitive free homotopy class. This is the place where primitivity comes into play in the definition of $\mathcal{P}^\circ$.
Moreover, these closed orbits are freely homotopic to the corresponding element of $\mathcal{C}_1$.

It remains to account for the sum of Lefschetz indices for the core orbit and possibly the boundary orbits.
For each $\gamma \in \mathcal{C}_1$, there are two cases:
\begin{enumerate}
    \item The prongs do not rotate along $\gamma$ . Then for $\phi^\circ$, the Lefschetz index of $\gamma$ is $1-k$, where $k$ is the number of prongs at $\gamma$. Meanwhile, for $\phi^\circ_\bullet$, the Lefschetz index of the core orbit of the blow-up region $\nu_\gamma$ is $+1$, and that of the $k$ closed orbits on $\partial \nu_\gamma$ are $-1$.
    \item The prongs rotate along $\gamma$. Then for $\phi^\circ$, the Lefschetz index of $\gamma$ is $+1$. Meanwhile, for $\phi^\circ_\bullet$, the Lefschetz index of the core orbit of the blow-up region $\nu_\gamma$ is $+1$, and the closed orbits on $\partial \nu_\gamma$ have non-primitve homotopy class hence does not count. 
\end{enumerate}
In both cases, the sum of Lefschetz indices for $\phi^\circ$ and for $\phi^\circ_\bullet$ are the same, thus well-definedness of the sum for $\phi^\circ_\bullet$ follows from the corresponding statement for $\phi^\circ$, which we argued for in the first paragraph.
\end{proof}

\subsection{Birkhoff sections} \label{subsec_birkhoff}

Let $\phi$ be a flow on an oriented closed 3-manifold $M$.
A \textbf{partial Birkhoff section} for $\phi$ is an immersed cooriented surface with boundary $S$ in $M$ where
\begin{itemize}
    \item the interior of $S$ is embedded and positively transverse to $\phi$, and
    \item the boundary components of $S$ cover closed orbits of $\phi$.
\end{itemize}
A \textbf{Birkhoff section} for $\phi$ is a partial section $S$ where every flow line intersects $S$ in finite forward and backward time.

Observe that a partial Birkhoff section without any boundary components is just a transverse surface, and a Birkhoff section without any boundary components is just a global section (also known as a cross section in the literature).
Generalizing in this direction, a cooriented surface $S$ is \textbf{almost transverse} to a pseudo-Anosov flow $\phi$ if it is positively transverse to $\phi$ in the complement of the singular orbits.
See \Cref{fig:almosttransverse}.

\begin{figure}
    \centering
    \begingroup%
  \makeatletter%
  \providecommand\color[2][]{%
    \errmessage{(Inkscape) Color is used for the text in Inkscape, but the package 'color.sty' is not loaded}%
    \renewcommand\color[2][]{}%
  }%
  \providecommand\transparent[1]{%
    \errmessage{(Inkscape) Transparency is used (non-zero) for the text in Inkscape, but the package 'transparent.sty' is not loaded}%
    \renewcommand\transparent[1]{}%
  }%
  \providecommand\rotatebox[2]{#2}%
  \newcommand*\fsize{\dimexpr\f@size pt\relax}%
  \newcommand*\lineheight[1]{\fontsize{\fsize}{#1\fsize}\selectfont}%
  \ifx\svgwidth\undefined%
    \setlength{\unitlength}{86.59554129bp}%
    \ifx\svgscale\undefined%
      \relax%
    \else%
      \setlength{\unitlength}{\unitlength * \real{\svgscale}}%
    \fi%
  \else%
    \setlength{\unitlength}{\svgwidth}%
  \fi%
  \global\let\svgwidth\undefined%
  \global\let\svgscale\undefined%
  \makeatother%
  \begin{picture}(1,1.18801546)%
    \lineheight{1}%
    \setlength\tabcolsep{0pt}%
    \put(0,0){\includegraphics[width=\unitlength,page=1]{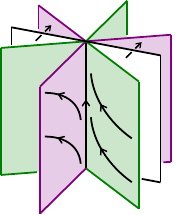}}%
    \put(0.91289428,0.1013445){\color[rgb]{0,0,0}\makebox(0,0)[lt]{\lineheight{1.25}\smash{\begin{tabular}[t]{l}$S$\end{tabular}}}}%
  \end{picture}%
\endgroup%

    \caption{For the purposes of this paper, a cooriented surface $S$ is almost transverse to a pseudo-Anosov flow $\phi$ if it is positively transverse to $\phi$ in the complement of the singular orbits.}
    \label{fig:almosttransverse}
\end{figure}

Now let $S$ be a partial Birkhoff section and $c$ be a boundary component of $S$.
If the orientation of $c$ as a closed orbit of $\phi$ agrees with its orientation as a boundary component of $S$, we say that $c$ is a \textbf{positive} boundary component. Otherwise we say that $c$ is a \textbf{negative} boundary component.
See \Cref{fig:birkhoffsectionboundary}.

\begin{figure} 
    \centering
    \selectfont \fontsize{14pt}{14pt}
    \begingroup%
  \makeatletter%
  \providecommand\color[2][]{%
    \errmessage{(Inkscape) Color is used for the text in Inkscape, but the package 'color.sty' is not loaded}%
    \renewcommand\color[2][]{}%
  }%
  \providecommand\transparent[1]{%
    \errmessage{(Inkscape) Transparency is used (non-zero) for the text in Inkscape, but the package 'transparent.sty' is not loaded}%
    \renewcommand\transparent[1]{}%
  }%
  \providecommand\rotatebox[2]{#2}%
  \newcommand*\fsize{\dimexpr\f@size pt\relax}%
  \newcommand*\lineheight[1]{\fontsize{\fsize}{#1\fsize}\selectfont}%
  \ifx\svgwidth\undefined%
    \setlength{\unitlength}{158.64531389bp}%
    \ifx\svgscale\undefined%
      \relax%
    \else%
      \setlength{\unitlength}{\unitlength * \real{\svgscale}}%
    \fi%
  \else%
    \setlength{\unitlength}{\svgwidth}%
  \fi%
  \global\let\svgwidth\undefined%
  \global\let\svgscale\undefined%
  \makeatother%
  \begin{picture}(1,0.71471281)%
    \lineheight{1}%
    \setlength\tabcolsep{0pt}%
    \put(0,0){\includegraphics[width=\unitlength,page=1]{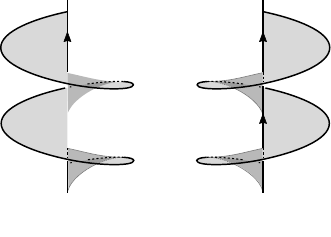}}%
    \put(0.16600881,0.01475452){\color[rgb]{0,0,0}\makebox(0,0)[lt]{\lineheight{1.25}\smash{\begin{tabular}[t]{l}$+$\end{tabular}}}}%
    \put(0,0){\includegraphics[width=\unitlength,page=2]{birkhoffsectionboundary.pdf}}%
    \put(0.75774131,0.0147537){\color[rgb]{0,0,0}\makebox(0,0)[lt]{\lineheight{1.25}\smash{\begin{tabular}[t]{l}$-$\end{tabular}}}}%
  \end{picture}%
\endgroup%

    \caption{A Birkhoff section near a positive/negative boundary component.} 
    \label{fig:birkhoffsectionboundary}
\end{figure}

We say that a partial Birkhoff section is \textbf{positive} if it only has positive boundary components.
We say that a pseudo-Anosov flow $\phi$ is \textbf{positive} if it admits a positive Birkhoff section.
We say that $\phi$ is \textbf{strictly positive} if it is positive and does not admit an almost transverse surface.

One defines the notions of \textbf{negative} and \textbf{strictly negative} symmetrically.
Building on this, we make the following definition.

\begin{defn}[Relatively positive/negative] \label{defn:relposflow}
Let $L$ be a link in an oriented closed 3-manifold $M$.
We say that a partial Birkhoff section is \textbf{positive} relative to $L$ if all of its negative boundary components lie along $L$, and \textbf{negative} relative to $L$ if all of its positive boundary components lie along $L$.
Note that we do not require that every component of $L$ be a boundary component of $M$.

We say that a pseudo-Anosov flow $\phi$ is \textbf{positive} relative to $L$ if $L$ is a collection of closed orbits of $\phi$ (under some choice of orientation) and $\phi$ admits a Birkhoff section that is positive relative to $L$.
We say that $\phi$ is \textbf{strictly positive} relative to $L$ if it is positive relative to $L$ and every almost transverse surface intersects $L$ transversely in at least one point.

One defines the notions of \textbf{negative} and \textbf{strictly negative} relative to $L$ symmetrically.
\end{defn}

For the rest of this section, we record two examples of relatively positive pseudo-Anosov flows. To that end, we recall the definition of the orbit space.

Given a pseudo-Anosov flow $\phi$ on $M$, let $\widetilde{\phi}$ be the lifted flow on the universal cover $\widetilde{M}$. By \cite[Proposition 4.1]{FM01}, the set of all flow lines of $\widetilde{\phi}$, equipped with the quotient topology, is a space $\mathcal{O}$ homeomorphic to the plane $\mathbb{R}^2$. 
We refer to $\mathcal{O}$ as the \textbf{orbit space} of $\phi$.
The orientation of $M$ induces a natural orientation on $\mathcal{O}$.
The lifted singular 2-dimensional foliations $\widetilde{\mathcal{F}^s}$ and $\widetilde{\mathcal{F}^u}$ induce singular 1-dimensional foliations $\mathcal{O}^s$ and $\mathcal{O}^u$ on $\mathcal{O}$.

\begin{eg}[Skew Anosov flow] \label{eg:skewanosovimpliesbas}
A pseudo-Anosov flow $\phi$ is a \textbf{positive skew} Anosov flow if its orbit space $\mathcal{O}$ is homeomorphic, via an orientation preserving homeomorphism, to a diagonal strip $\{(x,y) \mid -x < y < -x+1 \}$, where the foliations $\mathcal{O}^s$ and $\mathcal{O}^u$ correspond to the foliation by vertical and horizontal lines respectively.

If $\phi$ is a positive skew Anosov flow, then by \cite[Theorem A]{ABM24}, $\phi$ admits a positive Birkhoff section, thus $\phi$ is positive relative to any link. 

Moreover, such a $\phi$ cannot admit an almost transverse surface. Indeed, since there are no singular orbits, such a surface must be transverse, which is impossible in this case (see \cite[Proposition 4.9]{Bar95a}). 

Thus $\phi$ is in fact strictly positive relative to any link. 
\end{eg}

Given a pseudo-Anosov flow $\phi$, a \textbf{positive lozenge} is an orientation-preserving properly embedded copy of $[0,1]^2 \backslash \{(0,0), (1,1)\}$ in the orbit space $\mathcal{O}$, such that the restricted foliations $\mathcal{O}^s, \mathcal{O}^u$ are the foliations by vertical and horizontal lines respectively. 
Similarly, a \textbf{negative lozenge} is an orientation-preserving properly embedded copy of $[0,1]^2 \backslash \{(0,1), (1,0)\}$ in $\mathcal{O}$, such that the restricted foliations $\mathcal{O}^s, \mathcal{O}^u$ are by the foliations by vertical and horizontal lines respectively. 
The flow $\phi$ is said to have \textbf{no perfect fits} if there are no positive nor negative lozenges in $\mathcal{O}$.

Note that the usual definition of no perfect fits concerns perfect fit rectangles instead of lozenges. Nevertheless, the definition here is equivalent to the usual definition by \cite[Proposition 5.5]{Fen16}.

\begin{eg}[Pseudo-Anosov flows with no perfect fits] \label{eg:noperfectfitimpliesbas}
If $\phi$ has no perfect fits and is not the suspension of an Anosov diffeomorphism, then by \cite[Theorem 6.3]{Tsa24a}, $\phi$ admits a Birkhoff section whose boundary orbits are the singular orbits and one extra nonsingular orbit $\gamma$.
We deduce that if $\gamma$ is a positive (resp.~negative) boundary orbit, then $\phi$ is positive (resp.~negative) relative to $\sing(\phi)$.

Moreover, such a $\phi$ cannot admit an almost transverse surface that is disjoint from or tangent to the singular orbits. 
Indeed, as in \Cref{eg:skewanosovimpliesbas}, the intersection of $S$ with the stable foliation determines a nonsingular line field, thus $S$ is a torus.
But a pseudo-Anosov flow without perfect fits that is not the suspension of an Anosov diffeomorphism can only exist on atoroidal 3-manifolds (see, e.g., \cite[Chapter 5]{BM25}).
Thus $\phi$ is in fact strictly positive or negative relative to $\sing(\phi)$.

A closer inspection of the proof of \cite[Theorem 6.3]{Tsa24a} reveals that one can prescribe $\gamma$ to be positive or negative. Thus $\phi$ is actually both strictly positive and strictly negative relative to $\sing(\phi)$.
\end{eg}

\section{Reeb flow}

\subsection{Contact structures}

Throughout this paper, $Q$ will denote an oriented 3-manifold with boundary.

A 1-form $\alpha$ on $Q$ is positive contact if it satisfies $\alpha \wedge d\alpha > 0$.
A (positive) contact structure on $Q$ is a cooriented plane field $\xi$ given by the kernel of a (positive) contact 1-form. As we will only consider positive contact structures in this article, we will drop the qualifier positive from now on.

A contact vector field for a contact structure $\xi$ is a vector field $V$ that generates a flow preserving $\xi$.
The boundary $\partial Q$ is said to be convex if there exists a contact vector field $V$ transverse to it. 
From now on, unless otherwise stated, all contact structures on 3-manifolds with boundary will have convex boundary. This is always achievable by an arbitrarily small deformation.
In this case, $\partial Q$ can be divided as $R_+ \cup_\Gamma R_-$, where
\begin{itemize}
    \item $R_+$ is the subset of $\partial Q$ where $V$ is positively transverse to $\xi$, 
    \item $R_-$ is the subset of $\partial Q$ where $V$ is negatively transverse to $\xi$, and 
    \item $\Gamma$ is the subset of $\partial Q$ on which $V$ lies in $\xi$. 
\end{itemize}
It can be shown that the dividing set $\Gamma$ is a multicurve whose isotopy class only depends on the contact structure $\xi$ and not on the contact vector field $V$.

The Reeb vector field of a contact 1-form $\alpha$ is the unique vector field $R_\alpha$ satisfying $\alpha(R_\alpha) = 1$ and $d\alpha(R_\alpha,\cdot) = 0$. The Reeb flow of $\alpha$ is the flow generated by $R_\alpha$.

See, for instance, \cite{Mas14} for a more detailed introduction to contact structures and convex surfaces.

\begin{defn}[{\cite[Definition 2.8]{CGHH11}}]\label{def:adaptedcontactform}
Let $(Q,\xi)$ be a contact manifold with convex boundary and $\Gamma$ the dividing set on $\partial Q$. The triple $(Q,\Gamma,\xi)$ is called a \textbf{sutured contact manifold} if $(Q,\Gamma)$ is a sutured manifold, and there exists a contact form $\lambda$ for $\xi$ such that the following conditions hold:
    \begin{enumerate}
        \item The Reeb vector field $R_\lambda$ is positively transverse to $R_+$ and negatively transverse to $R_-$.
        \item $\lambda = Cdt + \beta$ in $N(\Gamma)$, where $\beta$ is independent of $t$. In particular, $R_\lambda = 1/C \partial_t$ in $N(\Gamma)$.
    \end{enumerate}
    Such a form $\lambda$ is said to be \textbf{adapted} to $(Q,\Gamma,\xi)$.
\end{defn}

A closed orbit of a flow is \textbf{nondegenerate} if the linearized first return map along this orbit does not have 1 as an eigenvalue. We say that a flow is nondegenerate if all of its closed orbits are nondegenerate. A contact structure is \textbf{hypertight} if it admits a defining 1-form whose Reeb flow has no nullhomotopic orbits. We also call such a defining 1-form a hypertight contact form.

Let $\xi_{ot}$ be the contact structure on $D^2 \times[0,\varepsilon]$ defined as $\ker(\cos(r\pi) dz + \sin(r\pi) d\theta)$. We say that a contact structure $(Q,\xi)$ is \textbf{tight} if there is no contact embedding $(D^2, \xi_{ot})\hookrightarrow (Q,\xi)$. Hofer proved that any hypertight contact structure on a closed contact 3-manifold is tight \cite[Theorem 1]{Hof93}. The proof works just as well for sutured contact 3-manifolds.

Let $\xi_{gt}$ be the contact structure on $T^2\times\mathbb{R}$ defined as $\ker(\cos(z)dx + \sin(z)dy)$. For any non-negative $n\in \mathbb R$, we say that a contact structure $(M,\xi)$ \textbf{contains $2\pi n$ Giroux torsion} if there is a contact embedding of $(T^2\times[0,2n\pi],\xi_{gt}) \hookrightarrow (Q,\xi)$, and we call the image of that embedding a \textbf{$2\pi n$ Giroux torsion domain}. We say that a contact structure contains Giroux torsion if it contains $2\pi$ Giroux torsion.

\begin{thm}[\cite{CGH09}]\label{thm:finiteness_contact_structure}
Let $Q$ be a compact 3-manifold, possibly with boundary and $\xi$ the germ of a contact structure along $\partial Q$. Then up to the action of diffeomorphisms isotopic to the identity rel boundary and Dehn twists along tori in the interior of $Q$, for any fixed $n$, there are a finite number of tight contact structures which are equal to $\xi$ near $\partial Q$ and which do not contain $2\pi n$ Giroux torsion.
\end{thm}

When building a Reeb flow from our BAS pseudo-Anosov flows, we will pass by a generalization of contact structure:
A \textbf{stable Hamiltonian structure} on $Q$ is a pair $(\omega,\lambda)$ where $\omega$ is a closed, nowhere vanishing 2-form and $\lambda$ is a 1-form such that 
\begin{itemize}
    \item $d\lambda = f\omega$ for some smooth function $f$
    \item $\omega \wedge \lambda > 0$ everywhere.
\end{itemize}
If $f>0$ everywhere, then $\lambda$ is a contact 1-form. For more background on stable Hamiltonian structures, see \cite{CV15}.

\subsection{Giroux torsion, cylindrical contact homology and homotopy classes of periodic orbits}\label{sec:contacthomology}
Cylindrical contact homology is a chain complex associated with a hypertight contact form. It is constructed in \cite{HN16} for closed 3-manifolds, but the construction works as well for sutured contact 3-manifolds admitting hypertight adapted contact forms \cite[Remark 1.6]{HN16}. Given a contact 3-manifold $(Q,\xi)$, a nondegenerate, hypertight contact form $\alpha$, and a free homotopy class $h$, we define $CC(Q,\alpha,h)$ to be the $\mathbb Q$ vector space freely generated by (possibly multiply covered) closed orbits of $R_\alpha$ in homotopy class $h$. There is a differential $\partial\colon CC(Q,\alpha,h)\to CC(Q,\alpha,h)$ which counts index 1 holomorphic cylinders in the symplectization of $Q$. This differential satisfies $\partial^2 =0$, making $CC(Q,\alpha,h)$ into a chain complex. We denote the homology of $CC(Q,\alpha,h)$ by $HC(Q,\alpha,h)$. We will only need the following formal properties of $CC$:
\begin{enumerate}
    \item 
    
    $CC$ is $\mathbb Z/2$ graded by Lefschetz index of closed orbits. The differential decreases this index by 1. When $h$ contains finitely many orbits of the flow, it follows that $$\chi(HC(Q,\alpha,h)):=rk(HC_{even}(Q,\alpha,h))-rk(HC_{odd}(Q,\alpha,h))$$ is the sum of the Lefschetz indices of closed orbits in homotopy class $h$. If $h$ contains infinitely many closed orbits of $R_\alpha$ and $\chi(HC(Q,\alpha,h))$ is finite, then $h$ contains infinitely many orbits of Lefschetz index 1 and of index -1.
    \item $HC$ depends only on $(Q,\xi)$, not on the choice of hypertight, nondegenerate $\alpha$.
    \item If $HC(Q,\alpha,h)\ncong 0$, then $R_\alpha$ has a closed orbit in the free homotopy class $h$.
\end{enumerate}

In \cite{ChaPan25}, Chaidez--Pan define a generalization of cylindrical contact homology, denoted by $CH$, to contact forms which are not necessarily hypertight. They say that a sutured contact 3-manifold is \textbf{algebraically tight} if its full contact homology vanishes. We will not need to work directly with this definition; we will need only the fact that a hypertight contact structure is algebraically tight \cite[Lemma 4.17]{ChaPan25}.

It turns out that the cylindrical contact homology can also detect the (lack of) Giroux torsion. More precisely, we will use the following consequence of a result of Chaidez--Pan:
\begin{prop}[Chaidez--Pan \cite{ChaPan25}] \label{prop_Girouxtorsion_implies_lotsorbits}
    Suppose $(Q,\Gamma, \xi)$ is a sutured contact manifold, with $\xi$ algebraically tight, $\alpha$ a contact form for $\xi$, and let $T$ be the boundary of a $2\pi n$ Giroux torsion domain, with $n\geq 1$, then for each homotopy class $\gamma \in \pi_1(T)$, either $\gamma$ or $\gamma^{-1}$ is represented by a closed orbit of the Reeb flow of $\alpha$.
\end{prop}
\begin{proof}
    The result follows directly from Theorem 4.27 and Lemma 4.19 of \cite{ChaPan25}, both of which work in the setting of contact sutured manifolds (see Remark 9 of \cite{ChaPan25}).
\end{proof}

\section{From pseudo-Anosov flow to Reeb flow} \label{sec:patoreeb}

The key to proving the main theorem is the following proposition, which transplants the dynamics of a pseudo-Anosov flow into the dynamics of a Reeb flow.

\begin{prop} \label{prop:patoreeb}
Let $\phi$ be a pseudo-Anosov flow on an oriented closed 3-manifold $M$ that is strictly positive with respect to a link $L$.
For every choice of essential multicurve $\Gamma$ on the boundary of $M^\circ := M \backslash \nu(L)$ 
with slope different from the degeneracy slope, 
there is a contact structure $\xi$ on $M^\circ$ with a contact 1-form $\alpha$ whose Reeb flow $X$ satisfies the following properties: 
\begin{enumerate}[label=(\arabic*)]
    \item\label{item_patoreeb_nondegenerate} $X$ is nondegenerate.
    \item\label{item_patoreeb_nonullhomotopic} Each closed orbit of $X$ is homotopically essential in $M^\circ$.
    In fact each closed obit $\gamma$ of $X$ is freely homotopic to a closed orbit of $\phi$ in $M$, hence also homotopically essential in $M$, except if $\gamma$ is homotopic to a meridian around some element of $L$. 
    \item \label{item_patoreeb_inacone} For each connected component $T$ of $\partial M^\circ$, the peripheral orbits of $X$ stay in a proper cone of $H_1(T)$.
    \item\label{item_patoreeb_samespectra} $\mathcal{P}^\circ(X)=\mathcal{P}^\circ(\phi^{\circ})$.
    \item\label{item_patoreeb_lefschetz} Moreover, for every $g \in \mathcal{P}^\circ(X)$, the sum of Lefschetz indices of closed orbits of $X$ with free homotopy class $g$ is well-defined (possibly $-\infty$) and nonzero. 
    \item\label{item_patoreeb_adapted} $\alpha$ is adapted to the sutured manifold $(M^\circ, \Gamma)$ in the sense of \Cref{def:adaptedcontactform}.
    \item\label{item_patoreeb_stdform} In a tubular neighbourhood of $\partial M^\circ$, $\xi$ has a standard form depending only on $\Gamma$.
    \item\label{item_patoreeb_nogirouxtorsion} $\xi$ has no Giroux torsion. 
\end{enumerate}
\end{prop}

This section is devoted to the proof of \Cref{prop:patoreeb}. Here is an outline: 
By assumption of relative positivity, $\phi$ admits a Birkhoff section $S$ all of whose negative boundary components lie along $L$. Through $S$, we can apply the main construction of \cite{Zun25} to construct a stable Hamiltonian structure $(\omega,\lambda)$ whose Reeb flow $X_1$ is a solid blow-up of the boundary blow-up $\phi^\circ$. 
The fact that $S$ only has positive boundary components in the interior of $M^\circ$ translates to the fact that $\lambda \wedge d \lambda \geq 0$, with strict inequality near the positive boundary components.

The next step is to `diffuse' this positivity to the rest of the manifold. This can be done by adding a small multiple of a primitive of $\omega$ to $\lambda$. For this to be possible, one has to arrange for $\omega$ to be exact in the previous step. In turn, this relies on a homology argument, which we explain in \Cref{subsec:homologycone}. 
The fact that $\phi$ is relatively \emph{strictly} positive is used here.
After this diffusion operation, we now have a contact 1-form $\alpha_1$ with the same Reeb flow $X_1$ (up to reparametrization).

At this point, the remaining items to arrange for are
\ref{item_patoreeb_adapted}, concerning the adaptiveness to the sutured contact structure, 
\ref{item_patoreeb_stdform}, concerning the standard form at the boundary,
\ref{item_patoreeb_nondegenerate}, concerning the nondegeneracy of the Reeb flow, and
\ref{item_patoreeb_nogirouxtorsion}, concerning the lack of Giroux torsion.
We will arrange for the first three by modifying the contact form near the boundary of $M^\circ$, then perturbing inside the blow-up regions.
Finally, we will argue for the last item using an argument involving \Cref{prop_Girouxtorsion_implies_lotsorbits} and JSJ decomposition.

\subsection{Homology cone lemma} \label{subsec:homologycone}

Let $\psi$ be a flow on an oriented closed 3-manifold.
Following \cite{Fri82}, we define the \textbf{cone of homology directions} of $\psi$ to be the set $\cone_1(\psi) \subset H_1(M; \mathbb{R})$ generated as a cone by the homology classes of (oriented) closed orbits of $\psi$.

For the proof of \Cref{prop:patoreeb}, we need a lemma concerning the positioning of $L$ within $\cone_1(\phi)$.
To state it, let us define $\cone_1(L) \subset H_1(M; \mathbb{R})$ to be the cone generated by the homology classes of closed orbits in $L$.

\begin{prop} \label{prop:singhomologycone}
Let $\phi$ be a pseudo-Anosov flow on an oriented closed 3-manifold $M$ where every almost transverse surface intersects $L$ transversely in at least one point.
Then $\cone_1(L)$ contains a point in the relative interior of $\cone_1(\phi)$.
\end{prop}
\begin{proof}
Suppose otherwise, then $\cone_1(L)$ is contained in a lower dimensional face of $\cone_1(\phi)$, thus there exists $\alpha \in H_2(M;\mathbb{R}) \backslash \{0\}$ that has nonnegative pairing with every element in $\cone_1(\phi)$ and zero pairing with every element in $\cone_1(L)$.
By \cite[Theorem 1.3.2]{Mos92b}, there exists a surface $S$ almost transverse to $\phi$ and with homology class $\alpha$.

For every orbit $\gamma \subset L$, we have $\langle S, \gamma \rangle = \langle \alpha, [\gamma] \rangle = 0$. Thus $S$ is either disjoint or tangent to each orbit in $L$. 
Contradiction. 
\end{proof}

Generalizing the definition above, for a flow $\psi$ on an oriented 3-manifold with boundary $Q$, we define $\cone_1(\psi) \subset H_1(Q,\partial Q; \mathbb{R})$ to be the cone generated by the homology classes of closed orbits of $\psi$.

\begin{cor} \label{cor:homologycone}
Let $\phi$ be a pseudo-Anosov flow on an oriented closed 3-manifold $M$ where every almost transverse surface intersects $L$ transversely in at least one point.
Let $\phi^\circ$ be a boundary blow-up of $\phi$ at $L$. 
Then $0 \in H_1(M^\circ,\partial M^\circ; \mathbb{R})$ lies in the relative interior of $\cone_1(\phi^\circ)$.
\end{cor}
\begin{proof}
From the definition, $\cone_1(\phi^\circ)$ is the image of $\cone_1(\phi)$ under the map 
$$H_1(M; \mathbb{R}) \to H_1(M,L; \mathbb{R}) \cong H_1(M^\circ,\partial M^\circ; \mathbb{R}).$$
Meanwhile, $\cone_1(L)$ is sent to $0$ under this map.
Thus the corollary follows from \Cref{prop:singhomologycone}.
\end{proof}

\subsection{Blowing up the flow}

For the rest of this section, we fix a pseudo-Anosov flow $\phi$ and choice of essential multicurve $\Gamma$ on $\partial M^\circ$ as in the hypothesis of \Cref{prop:patoreeb}. 
The first step in the proof of \Cref{prop:patoreeb} is to construct a stable Hamiltonian structure whose Reeb flow is a solid blow-up $\phi^\circ_{\bullet}$ of $\phi^\circ$.

\begin{lem} \label{lem:patoshs}
There exists a stable Hamiltonian structure $(\omega,\lambda)$ on $M^\circ$ whose Reeb flow $X_1$ satisfies the following properties:
\begin{enumerate}[label=(\roman*)]
    \item\label{item_patoshs_blowup} $X_1$ is a solid blow-up $\phi^\circ_{\bullet \mathcal{C}}$ for some collection $\mathcal{C}$ of closed orbits. Thus by \Cref{lem:blowupnonperispec}, we have:
    \begin{enumerate}
        \item Each closed orbit of $X_1$ is homotopically essential in $M^\circ$. In fact, each closed orbit of $X_1$ is freely homotopic to a closed orbit of $\phi$ in $M$, hence also homotopically essential in $M$.
        \item $\mathcal{P}^\circ(X_1) = \mathcal{P}^\circ(\phi^\circ)$.
        \item For every $g \in \mathcal{P}^\circ(X_1)$, the sum of Lefschetz indices of closed orbits of $X_1$ with free homotopy class $g$ is well-defined and nonzero.
    \end{enumerate}
    \item\label{item_patoshs_dlambda} $\lambda \wedge d \lambda \geq 0$. In other words, $d\lambda$ is a nonnegative multiple of $\omega$ at each point.
    \item\label{item_patoshs_zerohomology} $[\omega] = 0$ in $H^2(M^\circ;\mathbb{R})$.
    \item\label{item_patoshs_localcoordinates} Within the blow-up region around each boundary component of $M^\circ$, $(\omega,\lambda)$ are of the form
    \begin{align*}
    \omega &= dr d\theta - s(r) dr dz \\
    \lambda &= q d\theta + p dz
    \end{align*}
    in cylindrical coordinates, where $s$ is a monotone function that is a restriction of that in \Cref{eq:blowupreebflow}.
\end{enumerate}

\end{lem}
\begin{proof}
This lemma essentially follows from the main construction of \cite{Zun25}. For completeness, we will outline the main ideas of the construction here.

By assumption of relative positivity, $\phi$ admits a Birkhoff section $S$ all of whose negative boundary components lie along $L$. 
We solid blow-up $\phi$ at a collection $\mathcal{D}$ of closed orbits, containing $L \cup \partial S \cup \sing(\phi)$, to get a flow $\phi_{\bullet \mathcal{D}}$ that is smooth and for which $S$ remains a Birkhoff section.
Let $\nu_0$ be a union of $\phi_{\bullet \mathcal{C}}$-invariant tubular neighborhood around each component of $L \cup \partial S$, each contained in the corresponding blow-up region.
Then $\phi_{\bullet \mathcal{D}}$ restricts to a flow on $M_{\nu_0} = M \backslash \nu_0$ that has $S_{\nu_0} = S \cap M_{\nu_0}$ as a global section.
Correspondingly, we can express $M_{\nu_0}$ as a mapping torus $S_{\nu_0} \times [0,1]_t/(x,1) \sim (f_{\nu_0}(x),0)$ with monodromy $f_{\nu_0}\colon S_{\nu_0} \to S_{\nu_0}$.
Up to a smooth deformation of $\phi$, we may assume that $f_{\nu_0}$ preserves an area form $\Omega_{\nu_0}$ on $S_{\nu_0}$ that has the form $\Omega_{\nu_0} = dr d\vartheta$ in polar coordinates near $\partial S_{\nu_0}$.

We set $\lambda_{\nu_0} = dt$ on $M_{\nu_0}$, and set $\omega_{\nu_0}$ to be the closed 2-form on $M_{\nu_0}$ induced by $\Omega_{\nu_0}$ under the mapping torus structure.
Then $(\omega_{\nu_0}, \lambda_{\nu_0})$ is a stable Hamiltonian structure on $M_{\nu_0}$ whose Reeb flow $X_{\nu_0}$ is the suspension flow, which is a reparametrization of the restriction of the blown-up flow $\phi_{\bullet \mathcal{D}}$ to $M_{\nu_0}$, or in other words, a boundary blow-up $\phi^{\circ \mathcal{D}}$. 

Note that the coordinate $\vartheta$ near $\partial S_{\nu_0}$ is not preserved by $f_{\nu_0}$, instead, $f_{\nu_0}$ effects a shear $(r,\vartheta) \mapsto (r,\vartheta-\delta(r))$. 
Thus we define $\psi = \vartheta + \delta(r)t$, so that $(r,\psi,t)$ descends to a coordinate system on $M_{\nu_0}$ near $\partial M_{\nu_0}$ with $r \in [r_0,1], \psi \in \mathbb{R}/\mathbb{Z}, t \in \mathbb{R}/\mathbb{Z}$.
Under this coordinate system,
$(\omega_{\nu_0},\lambda_{\nu_0})$ is of the form 
\begin{align*}
\omega_{\nu_0} &= dr d\psi - \delta(r) dr dt \\
\lambda_{\nu_0} &= dt.
\end{align*} 

Now within these blow-up regions, the change in coordinates between $(r,\psi,t)$ in the mapping torus structure of $M_{\nu_0}$, and the cylindrical coordinates $(r,\theta,z)$ in $M$ is of the form $\psi = n \theta + m z, t = q \theta + p z$ for some $np-mq = 1$.
Thus we can write
\begin{align*}
\omega_{\nu_0} &= h(r) dr d\theta - k(r) dr dz \\
\lambda_{\nu_0} &= q d\theta + p dz 
\end{align*}
In particular, the Reeb vector field $X_{\nu_0}$ must be a positive multiple of $ k(r) \frac{\partial}{\partial \theta} + h(r) \frac{\partial}{\partial z}$.
Now recall that the flow generated by $X_{\nu_0}$ is a reparametrization of $\phi^{\circ \mathcal{D}}$, thus $h(r) > 0$ and $k(r) = s(r) h(r)$, where $s$ is the function in \Cref{eq:blowupreebflow}.

Then up to rescaling $\omega_{\nu_0}$, we can assume that  
\begin{align*}
\omega_{\nu_0} &= dr d\theta - s(r) dr dz \\
\lambda_{\nu_0} &= q d\theta + p dz.
\end{align*}

Finally, we extend $(\omega_{\nu_0},\lambda_{\nu_0})$ into a stable Hamiltonian structure $(\omega,\lambda)$ inside the components of $\nu$ around $\partial S \backslash L$. 
We can extend $\omega_{\nu_0}$ by simply setting $\omega = dr d\theta - s(r) dr dz$.
To extend $\lambda_{\nu_0}$, we use the hypothesis that each of $\partial S \backslash L$ is a positive boundary component of $S$,
and the degeneracy slope $s_1$ is nonnegative.
This implies that the constant $p,q$, appearing in the change of coordinates $(r,\psi,t) \leftrightarrow (r,\theta,z)$ near each components of $\partial S \backslash L$, must both be positive.
Thus, similarly to \cite[Lemma 3.4]{Zun25}, we can set $\lambda = F(r) d\theta + G(r) dz$, where $F\colon [0,r_0] \to \mathbb R$ satisfies
\begin{itemize}
    \item $F'(r) \geq 0$, 
    \item $F(r) = r^2$ near $r=0$, and
    \item $F(r) = q$ near $r=r_0$. %
\end{itemize}
and $G\colon [0,r_0] \to \mathbb R$ satisfies
\begin{itemize}
    \item $G'(r) = -F'(r)s(r)$,  
    \item $G(r) = p$ near $r=r_0$,  and
    \item $G(r) > 0$.
\end{itemize}
See \Cref{fig:birkhoffsectionslope} for a summary of the slopes in this construction.

\begin{figure}
    \centering
    \fontsize{8pt}{8pt}
    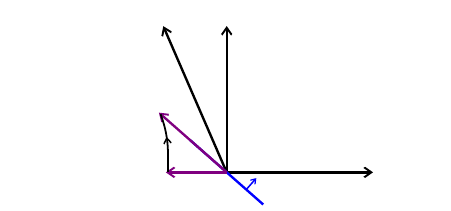
    \caption{Some slopes in \Cref{lem:patoshs}.}
    \label{fig:birkhoffsectionslope}
\end{figure}

Then the Reeb flow $X_1$ of $(\omega,\lambda)$ continues to be a reparametrization of the restriction of $\phi_{\bullet \mathcal{D}}$ to $M^\circ$, i.e., $X_1$ is a solid blow-up of $\phi^\circ$ along $\mathcal{C} = \mathcal{D} \backslash L$.
This shows item \ref{item_patoshs_blowup}. By construction $d\lambda = F'(r)\omega$ in the neighborhoods of $\mathcal{D} \backslash L$, so item \ref{item_patoshs_dlambda} follows, and similarly item \ref{item_patoshs_localcoordinates} also follows from the construction.

It remains to arrange for \ref{item_patoshs_zerohomology}. 
Note that there is freedom in the construction, coming from the choice of the orbits in $\mathcal{D} \backslash (L \cup \partial S \cup \sing(S))$ and the choice of the area form $\Omega_\nu$ within the blow-up regions around $\mathcal{D} \backslash (L \cup \partial S \cup \sing(S))$ (among other things).
These choices allow us to modify $[\omega] \in H^2(M^\circ; \mathbb{R}) \cong H_1(M^\circ, \partial M^\circ; \mathbb{R})$ by adding positive multiples of homology classes of closed orbits of $\phi^\circ$.
In particular, for every $x$ in the relative interior of $\cone_1(\phi^\circ)$, we can arrange for $[\omega]$ to be a large multiple of $x$.
By \Cref{cor:homologycone}, we can choose $x=0$.
\end{proof}

The next step is to `diffuse' $d\lambda$ by adding to $\lambda$ a primitive $\beta$ of $\omega$ so that it becomes a contact form $\alpha$.
To ensure that $\alpha$ is cylindrically symmetric near the boundary, we have to choose $\beta$ with cylindrical symmetry as well.

\begin{lem} \label{lem:symmetricprimitiveforomega}
Let $\omega$ be a closed 2-form on $M^\circ$ with $[\omega] = 0$ in $H^2(M^\circ;\mathbb{R})$. Suppose $\omega$ is of the form $\omega = dr d\theta - s(r) dr dz$ in the blow-up region around each boundary component of $M^\circ$. Then there exists a 1-form $\beta$ on $M^\circ$ such that $d\beta = \omega$ and such that $\beta$ is of the form $\beta = F(r) d\theta + G(r) dz$ within the blow-up region around each boundary component of $M^\circ$.
\end{lem}
\begin{proof}
Let $\beta_1 = \rho(r) (r d\theta - (\int_{r_0}^r s(u) du) dz)$, where $\rho$ is a bump function on $M^\circ$ such that $\rho(r)$ is constant equal to $1$ on $[r_0,1]$.
Then $d\beta_1 = \omega$ on $\nu$, thus $\omega - d\beta_1$ vanishes on $\nu$.

Under the exact sequence
$$H^1(\nu; \mathbb{R}) \to H^2(M^\circ, \nu; \mathbb{R}) \to H^2(M^\circ; \mathbb{R})$$
$[\omega - d\beta_1] \in H^2(M^\circ, \nu; \mathbb{R})$ maps to $[\omega - d\beta_1] = [\omega] = 0 \in H^2(M^\circ; \mathbb{R})$, thus $[\omega - d\beta_1]$ is the image of some class $z \in H^1(\nu; \mathbb{R})$.

But each class $x \in H^1(\nu; \mathbb{R})$ can be represented by a 1-form of the form $m d\theta + n dz$. 
The image of $[m d\theta + n dz] \in H^1(\nu; \mathbb{R})$ in $H^2(M^\circ, \nu; \mathbb{R})$ is $[d\beta_2]$ for $\beta_2 = \rho(r)(m d\theta + n dz)$, where $\rho$ is the same bump function as above.
Then $[\omega - d\beta_1 - d\beta_2] = 0$ in $H^2(M^\circ, \nu; \mathbb{R})$, thus there exists 1-form $\beta_3$ on $M^\circ$ vanishing on $\nu$ such that $\omega - d\beta_1 - d\beta_2 = d \beta_3$.

Setting $\beta = \beta_1 + \beta_2 + \beta_3$, we have $\omega = d \beta$ where $\beta = (r+m) d\theta - (\int_{r_0}^r s(r) dr - n) dz$ on $\nu$.
\end{proof}

With \Cref{lem:symmetricprimitiveforomega}, we can now perform the diffusion step.

\begin{lem} \label{lem:diffuseshs}
There exists a contact structure $\xi_1$ on $M^\circ$ with a contact 1-form $\alpha_1$ whose Reeb flow $X_1$ is (a reparametrization of) the one in \Cref{lem:patoshs}, and so that $\alpha_1$ is of the form $\alpha_1 = F(r) d\theta + G(r) dz$ within the blow-up region around each boundary component of $M^\circ$.
\end{lem}
\begin{proof}
Let $(\omega,\lambda)$ be the stable Hamiltonian structure defined in \Cref{lem:patoshs}.
By \Cref{lem:patoshs}, $d\lambda$ is a nonnegative multiple of $\omega$ at each point, say $d \lambda = f \omega$ for $f \geq 0$.
Let $\beta$ be a 1-form as in \Cref{lem:symmetricprimitiveforomega}. 
We set $\alpha_1 = \lambda + \epsilon \beta$.
Then $d\alpha_1 = (f+\epsilon) \omega$.
Thus $\alpha_1 \wedge d \alpha_1 = (f + \epsilon)(\lambda \wedge \omega + \epsilon \beta \wedge \omega)$.
Since $\lambda \wedge \omega > 0$, for small $\epsilon > 0$, we have $\lambda \wedge \omega + \epsilon \beta \wedge \omega > 0$. Thus $\alpha_1 \wedge d \alpha_1 > 0$.

Also, this computation shows that the Reeb vector field of $\alpha_1$ is a positive multiple of that of $(\omega,\lambda)$. So the Reeb flow of the former is a reparametrization of the latter.
\end{proof}

\subsection{Boundary modification and perturbation}

We are already most of the way towards proving \Cref{prop:patoreeb}.
The remaining items to arrange for are 
\ref{item_patoreeb_adapted}, saying that we can choose the contact form $\alpha$ to be adapted,
\ref{item_patoreeb_stdform}, saying that our contact structure has a standard form in the neighbourhood of $\partial M^\circ$,
\ref{item_patoreeb_nondegenerate}, saying that the associated Reeb flow is nondegenerate, and
\ref{item_patoreeb_nogirouxtorsion}, saying that there is no Giroux torsion.
We will arrange for the first three by modifying the contact form near the boundary of $M^\circ$, then perturbing inside the blow-up regions.
Finally, we will argue for the last item using an argument involving \Cref{prop_Girouxtorsion_implies_lotsorbits} and JSJ decomposition.

\begin{lem} \label{lem:dividingsetstdmodel}
There exists a contact structure $\xi_2$ on $M^\circ$ with a contact 1-form $\alpha_2$ whose Reeb flow $X_2$ satisfies the following properties:
\begin{enumerate}[label=(\roman*)]
    \item $X_2$ agrees with (a reparametrization of) $X_1$ away from a tubular neighborhood of $\partial M^\circ$ that is contained in the blow up region.
    \item Within the blow up region, $X_2$ is transverse to two transverse fibrations over $S^1$, whose positive half-planes intersect in a proper cone.
    \item The dividing set of $\xi_2$ on $M^\circ$ is $\Gamma$.
    \item $\alpha_2$ is adapted to the sutured manifold $(M^\circ, \Gamma)$ in the sense of \Cref{def:adaptedcontactform}. 
    \item In a tubular neighbourhood of $\partial M^\circ$, $\xi_2$ has a standard form depending only on $\Gamma$.
\end{enumerate}
In particular, $\alpha_2$ satisfies conditions \ref{item_patoreeb_nonullhomotopic}--\ref{item_patoreeb_stdform} in \Cref{prop:patoreeb}, and its Reeb flow $X_2$ is only degenerate within the blow-up regions around orbits in the interior of $M^\circ$, and regions homeomorphic to thickened tori $T^2 \times I$ contained in and parallel to boundary blow-up regions around boundary tori of $M^\circ$.
\end{lem}
\begin{proof}
Recall that within each blow-up region $\nu$ around a component of $\partial M^\circ$, the Reeb vector field $X_1$ of $\alpha_1$ is a positive multiple of $s(r) \frac{\partial}{\partial \theta} + \frac{\partial}{\partial z}$ at each point, where $s(r)$ is decreasing, and the limiting value of $s$ on $r=1$, the component of $\partial \nu$ that lies away from $\partial M^\circ$, is the degeneracy slope $s_1$ (recall from Construction \ref{const:blowup} that we assume $s_1<0$).
Let the slope of the given multicurve $\Gamma$ be $p \frac{\partial}{\partial \theta} - q \frac{\partial}{\partial z}$, where $p + qs_1 > 0$. 
Choose a complementary slope $-m \frac{\partial}{\partial \theta} + n \frac{\partial}{\partial z}$ so that $np-mq = 1$ and $m + ns_1 > 0$.
We make a change of coordinate $\psi = n \theta + m z$, $t = q \theta + p z$.
Then $X_1$ is a positive multiple of $\delta(r) \frac{\partial}{\partial \psi} + \frac{\partial}{\partial t}$, where $\delta(r):= \frac{s(r)n + m}{s(r) q+p}$ is positive and decreasing on $[r_0,1]$.

Meanwhile, the contact 1-form $\alpha_1$ is of the form $F(r) d\theta + G(r) dz$. 
Under this coordinate change, $\alpha_1$ can be written as $L(r) d\psi + K(r) dt$, where $L'(r) \delta(r) + K'(r) = 0$. 
We can now modify and extend $\alpha_1$ to $r \in (-\infty,r_0]$ so that $\alpha_1$ remains of the form $F(r) d\theta + G(r) dz$, and so that $\alpha_1 = (-\eta r+1) d\psi - r dt$ for $r \leq 0$, with $X_1 = \frac{\partial}{\partial \psi} - \eta \frac{\partial}{\partial t}$, 
for small $\eta > 0$.
See \Cref{fig:dividingsetslopes} for a summary of the slopes in this discussion.

\begin{figure}
    \centering
    \fontsize{8pt}{8pt}
    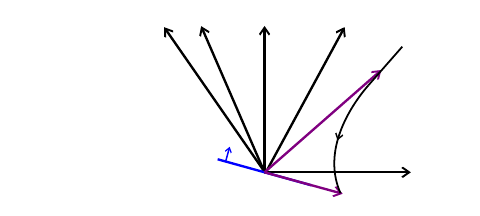
    \caption{Some slopes in \Cref{lem:dividingsetstdmodel}.}
    \label{fig:dividingsetslopes}
\end{figure}

We now construct a contact 1-form $\alpha_H$ on $[-1,0] \times T^2$ which we will use to glue onto $\alpha_1$ along $[-\epsilon,0] \times T^2$.
Consider the surface $A =  (\mathbb{R}/\mathbb{Z})_u \times [-1,0]_r$ equipped with the area form $\omega = du dr$, which has primitive $\beta = -r du$. 
For every function $H:A \to \mathbb{R}$, the Hamiltonian vector field $X_H$ on $A$ is defined by the equation $\iota_{X_H} \omega = dH$.
Cartan magic formula implies that $\mathcal{L}_{X_H} \beta = d(\iota_{X_H} \beta + H)$. 

Denoting by $\phi_\tau = (\phi^u_\tau,\phi^r_\tau): A \to A$ the time $\tau$-map of the flow of $X_H$, since $\phi_s^*\mathcal{L}_{X_H} \beta = \frac{d}{d\tau} \phi_\tau^*\beta|_{\tau=s}$, we get that 
\[
\phi_1^* \beta - \beta = \int_0^1 \frac{d}{d\tau} \phi_\tau^*\beta|_{\tau=s} ds = \int_0^1 \phi_s^*\mathcal{L}_{X_H} \beta ds = \int_0^1 d\phi^*_s\left(\iota_{X_H} \beta + H\right) ds = df,
\]
where $f:= \int_0^1 \phi^*_s\left(\iota_{X_H} \beta + H\right) ds$ is a function on $A$ (because integral of forms is a natural operation, see, e.g., \cite[Proposition 9.3.1]{MS17}).
We can now define a contact 1-form $\alpha_H = d\psi + \beta$ on the 3-manifold 
\begin{align*}
M_H &= \mathbb{R}_\psi \times A / (\psi,x) \sim (\psi-f(x),\phi_1(x)) \\
&= \left\{(\psi,u,r) \in \mathbb{R}_\psi \times A \mid \psi \in [0,f(u,r)]\right\}/ (f(u,r),u,r) \sim (0,\phi^u_1(u,r),\phi^r_1(u,r)),
\end{align*}
which has Reeb vector field $\frac{\partial}{\partial \psi}$.

Notice that while the coordinate $\psi$ descends to $M_H$, the same is not true for the coordinate $u$, since $\phi_1$ may not preserve $u$, i.e. the $u$-component $\phi^u_1(u,r)$ may not be equal to $u$.
To get around this, we set $t = (1-\frac{\psi}{f(u,r)}) u + \frac{\psi}{f(u,r)} \phi^u_1(u,r)$ on $\mathbb{R}_\psi \times (\mathbb{R}/\mathbb{Z})_u \times [-1,0]_r$.
The identification $(f(u,r),u,r) \sim (0,\phi^u_1(u,r),\phi^r_1(u,r))$ preserves $t$, so it descends to a function on $M_H$.

In this construction, if we take $H = \eta r$, then $\phi_1(u,r) = (u-\eta,r)$ and $df = 0$, so we can choose $f=1$,
in which case $M_H = \mathbb{R}_\psi \times (\mathbb{R}/\mathbb{Z})_u \times [-1,0]_r / (\psi,u,r) \sim (\psi-1,u-\eta,r)$ 
with $\alpha_H = d \psi - r du$.
Furthermore, $t = u - \eta \psi$, so we can rewrite
$\alpha_H = (-\eta r + 1) d \psi - r dt$.
On the other hand, if $(u_0,r_0)$ is a critical point of $H$, then it is a fixed point of $\phi_1$.
Along the closed curve $[0,f(u_0,r_0)] \times \{(u_0,r_0)\} / (f(u_0,r_0),u_0,r_0) \sim (0,u_0,r_0)$, $t$ is constant at $u_0$, and $r$ is constant at $r_0$, thus the curve is tangent to $\frac{\partial}{\partial \psi}$, which we recall is the slope of $\Gamma$.

Thus we take the function $H$ to have level sets as shown in \Cref{fig:dividingsetstdmodel}. 
More specifically, $H = \eta r$ in $r \in [-\epsilon,0]$, and $H$ has $N$ saddle critical points, where $2N$ is the multiplicity of $\Gamma$.
Within $r \in [-\epsilon,0]$, the contact forms $\alpha_1$ and $\alpha_2$ agree,
so we can glue the two together.
Near $r = -1$, we can find a convex surface $T$ by suspending the yellow curve in \Cref{fig:dividingsetstdmodel}. 
By truncating $M_H$ at $T$ then gluing to $M^\circ$, we obtain the desired contact form $\alpha_2$.

\begin{figure}
    \centering
    \begingroup%
  \makeatletter%
  \providecommand\color[2][]{%
    \errmessage{(Inkscape) Color is used for the text in Inkscape, but the package 'color.sty' is not loaded}%
    \renewcommand\color[2][]{}%
  }%
  \providecommand\transparent[1]{%
    \errmessage{(Inkscape) Transparency is used (non-zero) for the text in Inkscape, but the package 'transparent.sty' is not loaded}%
    \renewcommand\transparent[1]{}%
  }%
  \providecommand\rotatebox[2]{#2}%
  \newcommand*\fsize{\dimexpr\f@size pt\relax}%
  \newcommand*\lineheight[1]{\fontsize{\fsize}{#1\fsize}\selectfont}%
  \ifx\svgwidth\undefined%
    \setlength{\unitlength}{202.40436704bp}%
    \ifx\svgscale\undefined%
      \relax%
    \else%
      \setlength{\unitlength}{\unitlength * \real{\svgscale}}%
    \fi%
  \else%
    \setlength{\unitlength}{\svgwidth}%
  \fi%
  \global\let\svgwidth\undefined%
  \global\let\svgscale\undefined%
  \makeatother%
  \begin{picture}(1,0.66144968)%
    \lineheight{1}%
    \setlength\tabcolsep{0pt}%
    \put(0,0){\includegraphics[width=\unitlength,page=1]{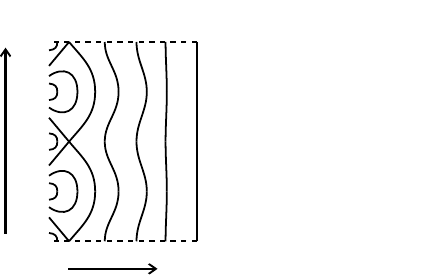}}%
    \put(0.13815301,0.62343013){\color[rgb]{0,0,0}\makebox(0,0)[lt]{\lineheight{1.25}\smash{\begin{tabular}[t]{l}$H$ (level set)\end{tabular}}}}%
    \put(-0.0023881,0.56988396){\color[rgb]{0,0,0}\makebox(0,0)[lt]{\lineheight{1.25}\smash{\begin{tabular}[t]{l}$u$\end{tabular}}}}%
    \put(0.39774422,0.00870873){\color[rgb]{0,0,0}\makebox(0,0)[lt]{\lineheight{1.25}\smash{\begin{tabular}[t]{l}$r$\end{tabular}}}}%
    \put(0,0){\includegraphics[width=\unitlength,page=2]{dividingsetstdmodel.pdf}}%
    \put(0.55900691,0.62343013){\color[rgb]{1,0.8,0}\makebox(0,0)[lt]{\lineheight{1.25}\smash{\begin{tabular}[t]{l}convex $\partial$\end{tabular}}}}%
    \put(0.81554824,0.62343013){\color[rgb]{0.50196078,0,0.50196078}\makebox(0,0)[lt]{\lineheight{1.25}\smash{\begin{tabular}[t]{l}$X_H$\end{tabular}}}}%
  \end{picture}%
\endgroup%

    \caption{Defining the gluing block $(M_H,\alpha_H)$.}
    \label{fig:dividingsetstdmodel}
\end{figure}

Within the blow-up region, the Reeb vector field $X_2$ is transverse to the fibrations given by $\psi$ and $t + 2 \eta \psi$ respectively.
The positive half-planes of these fibrations intersect in a proper cone, see the green co-vectors in \Cref{fig:dividingsetslopes}.
Thus \ref{item_patoreeb_nonullhomotopic}--\ref{item_patoreeb_lefschetz} for $X_2$ follow from the corresponding properties for $X_1$, which were established in \Cref{lem:patoshs}.
By construction, \ref{item_patoreeb_adapted} holds for $X_2$, and the contact form $\xi_2$ depends only on $\Gamma$ in a neighborhood of $\partial M^\circ$.
Finally, since $X_1$ is a blow up of $\phi^\circ$, $X_1$ is only degenerate within the blow-up regions, by construction, $X_2$ is degenerate only within the same regions, as well as the regions where the level sets of $H$ are closed curves parallel to the boundary components of $A$.
\end{proof}

Using the contact structure $\xi_2$, we can now prove \Cref{prop:patoreeb}.

\begin{proof}[Proof of \Cref{prop:patoreeb}]
Let $\alpha_2$ be the contact 1-form constructed in \Cref{lem:dividingsetstdmodel}.
We can do a generic perturbation supported within the regions where $X_2$ is degenerate, which are contained in the blow-up regions by the last sentence of \Cref{lem:dividingsetstdmodel}, in order to get a contact 1-form $\alpha$ for $\xi_2$ with Reeb flow $X$ that is nondegenerate, i.e., satisfies \ref{item_patoreeb_nondegenerate}.
See \cite[Theorem 13]{ABW10}. 

The flows $X_2$ and $X$ agree outside of blow-up regions. Within a blow-up region around a component of $\partial M^\circ$, the perturbed flow $X$ remains transverse to two fibrations, 
thus does not have nullhomotopic closed orbits.
Moreover, any closed orbit in such a blow-up region is homotopic to the corresponding element of $L$, unless if it is homotopic to the meridian.
Finally, any such closed orbit is peripheral, hence does not contribute towards $\mathcal{P}^\circ(X)$.

Similarly, within a blow-up region in the interior of $M^\circ$, $X$ remains transverse to the fibration by $z$ in \Cref{eq:blowupreebflow}, thus does not have nullhomotopic closed orbits.
Moreover, such closed orbits are homotopic to the corresponding orbit of $\phi$.
The sum of Lefschetz indices of closed orbits homotopic to the core of the blow-up region remains the same, thus is nonzero.
In particular, there exists such an orbit.
Any other closed orbits within such a blow-up region is nonprimitive, hence does not contribute towards $\mathcal{P}^\circ(X)$.

This analysis shows that \ref{item_patoreeb_nonullhomotopic}--\ref{item_patoreeb_lefschetz} for $X$ follow from the corresponding properties for $X_2$, which are established in \Cref{lem:dividingsetstdmodel}.

\ref{item_patoreeb_adapted} and \ref{item_patoreeb_stdform} both hold since the perturbation we did on $X_2$ to obtain $X$ is done away from the boundary components of $M^\circ$.

Finally, we will deduce \ref{item_patoreeb_nogirouxtorsion}:
Suppose for the sake of contradiction that $\xi$ has Giroux torsion along a torus $T \subset M^\circ$.
By the first part of \Cref{item_patoreeb_nonullhomotopic}, $\xi$ is hypertight, hence tight, so $T$ must be $\pi_1$-injective in $M^\circ$.
If $T$ is also $\pi_1$-injective in $M$, then by \Cref{prop_Girouxtorsion_implies_lotsorbits}, we have a subgroup $\mathbb{Z}^2 \leq \pi_1 M$ where for every $\gamma \in \mathbb{Z}^2$, either $\gamma$ or $\gamma^{-1}$ is freely homotopic to a closed orbit of $\phi$. This is impossible by \cite[Corollary 4.2.4]{BM25}.

On the other hand, note that if $T$ is not $\pi_1$-injective in $M$, then any element $g \in \ker(\pi_1 T \to \pi_1 M)$ must be a meridian of some element of $L$.
Indeed, by \Cref{prop_Girouxtorsion_implies_lotsorbits}, up to taking an inverse, $g$ is represented by a closed orbit of $X$, so the conclusion follows from the second part of \Cref{item_patoreeb_nonullhomotopic}.
In particular, this implies that $g$ lies in some peripheral subgroup of $\pi_1 M^\circ$.

We now divide into two cases:
Case 1 is if $\pi_1 T$ intersects a peripheral subgroup $P$ in a subgroup of rank one. In this case, consider the JSJ decomposition of $M^\circ$. The torus $T$ and the boundary component corresponding to $P$ must lie in the same Seifert piece $N$, and the intersection $\pi_1 T \cap P$ must be generated by the Seifert fiber of $N$.
But then the 3-manifold $M^\circ \cup \nu(\gamma)$ obtained by drilling $M$ at every orbit of $L$ except the orbit $\gamma$ corresponding to $P$ is reducible, contradiction.
Case 2 is if for every peripheral subgroup $P$ where $\pi_1 T \cap P$ is nontrivial, $\pi_1 T \cap P$ has rank two.
Since $M^\circ$ cannot be $T^2 \times I$, or else $M$ is a lens space (or $S^2 \times S^1$) and cannot admit a pseudo-Anosov flow, $\pi_1 T \cap P$ can only be nontrivial for one peripheral subgroup $P$, and $P$ is the peripheral subgroup associated to a unique boundary component.
In this case, $T$ is parallel to this boundary component. But by \Cref{item_patoreeb_inacone}, the peripheral orbits of $X$ stay in a proper cone, which contradicts \Cref{prop_Girouxtorsion_implies_lotsorbits}.\qedhere

\end{proof}

\section{Proof of main theorems}

\subsection{Finiteness of relatively strictly positive flows}

We are now ready to prove our main theorem.

\begin{thm} \label{thm:relposfinite}
Let $M$ be an oriented closed 3-manifold and let $L$ be a link in $M$.
Then up to orbit equivalence, there are only finitely many distinct pseudo-Anosov flows on $M$ that are strictly positive relative to $L$.

In fact, if $N$ is the number of Seifert pieces in the JSJ decomposition of $M$, then
the number of such orbit
equivalence classes is bounded from above by $2^N$ times the number of Giroux torsion free tight contact structures on $M^\circ = M \backslash \nu(L)$, up to diffeomorphisms isotopic to identity rel boundary, with a certain given boundary germ.

\end{thm}
\begin{proof}
Let $(\phi_i)$ be an infinite collection of such orbit equivalence classes of pseudo-Anosov flows. 
Up to restricting to an infinite subcollection,
we can pick an essential multicurve $\Gamma$ with slope different from all the degeneracy slopes of $\phi_i$.
Then we can apply \Cref{prop:patoreeb} to obtain the associated contact structures $\xi_i$ on the same sutured manifold $(M^\circ,\Gamma)$.

Let $X_i$ be the Reeb flow of $\alpha_i$ associated to $\phi^\circ_i$ given by Proposition \ref{prop:patoreeb}.
By the first part of item \ref{item_patoreeb_nonullhomotopic}, $\xi$ is hypertight and by item \ref{item_patoreeb_nogirouxtorsion}, $\xi$ has no Giroux torsion.
Also, by item \ref{item_patoreeb_stdform}, the contact structures $\xi_i$ all agree in a neighbourhood of $\partial M^\circ$. 
Thus, Colin--Giroux--Honda's finiteness result (Theorem \ref{thm:finiteness_contact_structure}) implies that 
up to taking a further subsequence, for each $i$, there exists a diffeomorphism $h_i\colon M^\circ \to M^\circ$ which is isotopic to the identity and satisfies $\xi_0 = h_i^*\xi_i$. 

This implies that the cylindrical contact homology of $\alpha_0$ and $h_i^*\alpha_i$ are the same. Moreover, item \ref{item_patoreeb_lefschetz} implies that for any $g\in \cP^\circ(X_i)$, the cylindrical contact homology $HC(M^\circ,\alpha_i,g) \neq 0$  (see Section \ref{sec:contacthomology}). Therefore, $\cP^\circ (X_0) = h_i^\ast \cP^\circ ( X_i )$.
By item \ref{item_patoreeb_samespectra}, we deduce that $\cP^\circ(\phi_0^\circ) = h_i^\ast\cP^\circ( \phi_i^\circ)$. 
Since $h_i$ is the identity on $\partial M^\circ$, we can extend $h_i$ to $M$, so that $h_i^\ast\cP^\circ( \phi_i^\circ) = \cP^\circ( (h_i^\ast\phi_i)^\circ)$.
Finally, Corollary \ref{cor:nonperispecdeterminespaflow} applies to show that some $h_i^\ast\phi_i$ and $h_j^\ast\phi_j$ are isotopically equivalent. This contradicts the assumption on the $\phi_i$.

In fact, this argument shows that there are at most as many such orbit equivalence classes of pseudo-Anosov flows as $2^N$ times the number of Giroux torsion free tight contact structures on $Q$, up to diffeomorphisms isotopic to identity rel boundary, with boundary germ given by that in \ref{item_patoreeb_stdform}.
\end{proof}

\subsection{Finiteness of veering triangulations}

In this subsection, we explain how to deduce finiteness of veering triangulations from \Cref{thm:relposfinite}.

Since it does not play a role in this paper, we will omit the definition of veering triangulations. 
It suffices to say that these are ideal triangulations $\Delta$ of (the interior of an) orientable compact 3-manifold $Q$ with torus boundary components satisfying certain combinatorial conditions. Furthermore, $\Delta$ has an associated, combinatorially defined, multicurve on each boundary component of $Q$, which we refer to as its \textbf{ladderpole} curve.
For more information see \cite[Chapter 1]{Tsathesis}.

The relevant fact to this paper is that there is a correspondence between veering triangulations and pseudo-Anosov flows, in the sense of \Cref{thm:correspondence} below.

\begin{thm}[Agol-Guéritaud, Schleimer-Segerman] \label{thm:correspondence}
Let $\mathcal{F}$ denote the set of triples $(M,\phi,\mathcal{C})$, where
\begin{itemize}
    \item $M$ is an orientable closed 3-manifold,
    \item $\phi$ is a pseudo-Anosov flow on $M$, and
    \item $\mathcal{C}$ is a collection of closed orbits of $\phi$ such that $\phi$ has no perfect fits relative to $\mathcal{C}$,
\end{itemize}
modulo isotopy equivalence.

Let $\mathcal{T}$ denote the set of triples $(Q,\Delta,s)$, where 
\begin{itemize}
    \item $Q$ is an orientable compact 3-manifold with torus boundary components,
    \item $\Delta$ is a veering triangulation on $Q$, and
    \item $s$ is a collection of slopes on the boundary components of $Q$ that intersects the ladderpole curve $l$ of $\Delta$ at least twice on each boundary component.
\end{itemize}
modulo isotopy.

Then:
\begin{enumerate}
    \item There exists a function $\mathsf{V}:\mathcal{F} \to \mathcal{T}$ such that 
    $$\mathsf{V}(M, \phi, \mathcal{C}) = (M \backslash \nu(\mathcal{C}), \Delta(N, \phi, \mathcal{C}), \text{meridians})$$
    Furthermore, the ladderpole curve at a boundary component of $M \backslash \nu(\mathcal{C})$ is the degeneracy curve at the corresponding orbit of $\mathcal{C}$.
    \item There exists a function $\mathsf{P}:\mathcal{T} \to \mathcal{F}$ such that 
    $$\mathsf{P}(Q, \Delta, s) = (Q(s), \phi(Q, \Delta, s), \text{cores of filling solid tori}).$$
    We refer to the cores of filling solid tori as the \textbf{filled orbits}. 
    The singular orbits of $\phi(Q, \Delta, s)$ must be among the filled orbits, and in fact the number of prongs at a filled orbit is the intersection number between $s$ and $l$ at the corresponding boundary component of $Q$.
    \item $\mathsf{V}$ and $\mathsf{P}$ are inverse to each other.
\end{enumerate}
\end{thm}

The proof of \Cref{thm:correspondence} will be contained in \cite{SS20, SS24, FSS25, SS23, SSpart5}.
See also \cite[Chapter 2]{Tsathesis} for an alternate proof.

Here, a pseudo-Anosov flow having no perfect fits relative to a collection of closed orbits is a generalization of the no perfect fits condition in \Cref{eg:noperfectfitimpliesbas}. For the purposes of this paper, it suffices to say that a pseudo-Anosov flow $\phi$ has no perfect fits if and only if it has no perfect fits relative to $\sing(\phi)$.

To deduce finiteness of veering triangulations, we need the following elementary lemma.

\begin{lem} \label{lem:singularfillingslopeexists}
Let $(l_n)$ be an infinite (possibly repeating) sequence of multicurves on a torus $T$. Then there exists a slope $s$ on $T$ and a subsequence $(l_{n_k})$ such that the intersection number between $s$ and $l_{n_k}$ is at least $3$ for each $k$.
\end{lem}
\begin{proof}
Pick a basis $\{\lambda, \mu\}$ on $T$, and write $l_n = a_n \lambda + b_n \mu$.
Assume the lemma is false, that is, for every slope $s$, all but finitely many $l_n$ intersect $s$ at most $2$ times.
Then applying this to $s = \pm 3 \lambda \pm \mu$ shows that $3|a_n| + |b_n| \leq 2$ for all but finitely many $n$.
That is, for all but finitely many $n$, $l_n = \mu$.
But similarly, applying this to $s = \pm \lambda \pm 3 \mu$ shows that for all but finitely many $n$, $l_n = \lambda$. Contradiction.
\end{proof}

\begin{thm}\label{thm_finiteness_veering}
Let $Q$ be an orientable compact 3-manifold with torus boundary components. Then there are at most finitely many veering triangulations on $Q$ up to isotopy.

In fact, the number of such isotopy classes is bounded from above by the number of Giroux torsion free tight contact structures on $Q$, up to diffeomorphisms isotopic to identity rel boundary, with a certain given boundary germ.
\end{thm}
\begin{proof}
Suppose otherwise that there is an infinite sequence of non-isotopic veering triangulations $(\Delta_n)$ on $M^\circ$. Applying \Cref{lem:singularfillingslopeexists} to the corresponding sequence of ladderpole curves on each boundary component, and passing to a subsequence, we get a collection of slopes $s$ on the boundary components of $M^\circ$ so that $s$ intersects the ladderpole curves of $\Delta_n$ at least three times on each boundary component, for each $n$.

By \Cref{thm:correspondence}, we then get an infinite sequence of non-isotopically equivalent pseudo-Anosov flow $\phi_n = \mathsf{P}(M^\circ, \Delta_n, s)$ on $M^\circ(s)$ that have no perfect fits relative to the collection of filled orbits. 
But by our intersection number condition, the filled orbits are exactly the singular orbits, thus each $\phi_n$ has no perfect fits.

By \Cref{eg:noperfectfitimpliesbas}, the flows $\phi_i$ are strictly positive relative to the singular orbits, thus to the filled orbits. The isotopy class of the latter only depends on $M^\circ$ and $s$, and not on the individual flows $\phi_i$. 
Thus we get a contradiction to \Cref{thm:relposfinite}. 

For the precise bound, we need the fact that a manifold admitting a veering triangulation must be hyperbolic, which is shown in \cite{HRST11}.
\end{proof}

\subsection{Finiteness of BAS flows}

In this subsection, we explain how \Cref{thm:relposfinite}, when combined with a result of Li, shows finiteness of pseudo-Anosov flows with no perfect fits.
In fact, we get something slightly stronger.

\begin{defn} \label{defn:BAS}
We say that a pseudo-Anosov flow $\phi$ is \textbf{$+$BAS} %
(\textbf{positive} \textbf{B}irkhoff section \textbf{a}way from \textbf{s}ingularities) if it is positive relative to its singular orbits.
Similarly, we say that a pseudo-Anosov flow $\phi$ is \textbf{$-$BAS} if it is negative relative to its singular orbits.
We say that $\phi$ is \textbf{BAS} if it is $+$BAS or $-$BAS.
\end{defn}

\begin{thm}[{\cite{Li26}}] \label{thm_finitelymanySingularityclass}
For a given atoroidal 3-manifold $M$, there are at most finitely many possibilities for the isotopy class of $\sing(\phi)$ and the degeneracy curves for a pseudo-Anosov flow $\phi$ on $M$.
\end{thm}

\begin{thm}\label{thm:finitenessBAS}
Let $M$ be a closed oriented atoroidal manifold. Up to isotopy equivalence, there are only finitely many distinct BAS pseudo-Anosov flows on $M$.
\end{thm}
\begin{proof}
Let $(\phi_i)$ be an infinite collection of pairwise isotopically inequivalent BAS pseudo-Anosov flows on $M$. Up to considering an infinite subfamily and/or switching the orientation of the manifold, we may assume that they are all +BAS.
By Li's finiteness result (Theorem \ref{thm_finitelymanySingularityclass}), up to taking a further subsequence, we can assume that all the $\sing(\phi_i)$ are the same link $L$.

By definition, each $\phi_i$ is positive relative to $L$.
We further claim that they are strictly positive.
Indeed, since $L = \sing(\phi_i)$, any almost transverse surface must be a torus, which contradicts the atoroidal assumption.

Thus we have an infinite collection of pairwise isotopically inequivalent pseudo-Anosov flows on $M$ that are strictly positive relative to $L$, contradicting \Cref{thm:relposfinite}.
\end{proof}

\begin{rmk}
Essential tori are a common tripping hazard in this paper; sometimes we can step over them carefully, and other times they force us to take longer detours. \Cref{thm_finitelymanySingularityclass} requires the 3-manifold to be atoroidal, and this ultimately traces back to the failure of Gabai's finiteness theorem for essential branched surfaces in the toroidal case.    
The assumption that $M$ is atoroidal is not sufficient to rule out Giroux torsion; we also need to rule out boundary-parallel Giroux torsion in $M^\circ$. 
Finally, the assumption that $M$ is atoroidal is sufficient to rule out almost transverse surfaces in $M\setminus \sing(\phi)$.
\end{rmk}

We can now deduce finiteness of pseudo-Anosov flows with no perfect fits.

\begin{thm}\label{thm:noperfectfitfinite}
Let $M$ be a closed 3-manifold. Then there are at most finitely many pseudo-Anosov flows without perfect fits on $M$ up to isotopy equivalence.
\end{thm}
\begin{proof}
Suppose there exist infinitely many pairwise non isotopically equivalent pseudo-Anosov flows $\phi_i$ on a $3$-manifold $M$. By Proposition \ref{prop_finite_cover}, we can assume that $M$ is orientable. First let us assume that one of the $\phi_i$ is a suspension of an Anosov diffeomorphism, then $M$ is the mapping torus of some hyperbolic matrix $A\in SL(2,\mathbb R)$, and it is already known that the only pseudo-Anosov flows on $M$ are isotopically equivalent to the suspension of $A$ or $A^{-1}$ (see \cite[Theorem 5.7]{BF13}.\footnote{Note that the statement in \cite[Theorem 5.7]{BF13} only says orbit equivalent, but one can verify that the orbit equivalence can be chosen isotopic to the identity.})

So we can now assume that all the $\phi_i$ are not Anosov flows. 
By Example \ref{eg:noperfectfitimpliesbas} all the $\phi_i$ are BAS flows. Furthermore, (non-Anosov) pseudo-Anosov flows without perfect fits only exist on atoroidal $3$-manifolds (see, e.g. \cite[Chapter 5]{BM25}). By Theorem \ref{thm:finitenessBAS} infinitely many of the $\phi_i$ must be isotopically equivalent, contradicting the assumption.
\end{proof}

\section{Discussion and further questions} \label{sec:questions}

\subsection{Orbit space characterization of BAS}

It is unclear to us how large the class of BAS pseudo-Anosov flows are.
There certainly exists non-BAS flows, which we show using the following obstruction:

\begin{prop} \label{prop:+BASimpliesnonegativepartialsection}
Let $\phi$ be a $+$BAS pseudo-Anosov flow. Then $\phi$ does not admit a strictly negative partial Birkhoff section that is disjoint from the singular orbits of $\phi$.
\end{prop}
\begin{proof}
By assumption of $+$BAS, $\phi$ admits a Birkhoff section $S$ all of whose negative boundary components lie along $\sing(\phi)$. Suppose for the sake of contradiction that $\phi$ also admits a negative partial section $T$ that is disjoint from $\sing(\phi)$.

Let $M^\circ$ be the complement of a tubular neighborhood of $\sing(\phi)$. The surfaces $S$ and $T$ induce surfaces $S^\circ$ and $T^\circ$ in $M^\circ$. Note that $S^\circ$ possibly intersects $\partial M^\circ$ while $T^\circ$ is disjoint from $\partial M^\circ$.
Now recall the definition of linking number between multiorbits in $M^\circ$: Let $a$ be a multiorbit that is nullhomologous rel boundary, and let $b$ be a multiorbit that is nullhomologous. Then $a$ bounds a surface $A^\circ$
along with some curves on $\partial M^\circ$, and $b$ bounds a surface $B^\circ$. 
Assuming that $a$ and $b$ do not share components, then the linking number between $a$ and $b$ is 
$$\lk(a,b) = \langle A^\circ, b \rangle = \langle a, B^\circ \rangle$$
where $\langle \cdot, \cdot \rangle$ denotes the algebraic intersection number.
If $a$ and $b$ do share components, we take 
$$\lk(a,b) = \langle A^\circ, b^* \rangle = \langle a^*, B^\circ \rangle$$
where $a^*$ and $b^*$ are the push-offs of $a$ and $b$ along their degeneracy slope with suitable weights. See \cite[Section 5]{Mar25}.

Applying this definition with $a$ being the boundary components of $S$ lying away from $\sing(\phi)$, and $b$ being the boundary components of $T$, we have
$$0 < \langle S^\circ, b^* \rangle = \lk(a,b) = \langle a^*, T \rangle \leq 0.$$
Here the first inequality is strict because $S$ is a Birkhoff section and $b$ is non-empty. This gives us the desired contradiction.
\end{proof}

\begin{noneg}[Flows with both positive and negative lozenges disjoint from singular orbits] \label{noneg:positiveandnegativelozenges}
Suppose a pseudo-Anosov flow $\phi$ has a positive lozenge $L$ that is disjoint from $\sing(\phi)$. Then \cite[Theorem 3.3.1]{BM25} implies that $\phi$ has a positive lozenge $L'$ fixed by $g \in \pi_1(M)$, and such that $L$ and $L'$ share an ideal corner. The latter property ensures that $L'$ is disjoint from $\sing(\phi)$ as well.

Then by \cite[Theorem 5.2.18]{BM25}, there exists an \textbf{immersed Birkhoff annulus} $A$ whose trace is equal to $L'$. Here, an immersed Birkhoff annulus is a partial section that is homeomorphic to an annulus, but whose interior is not necessarily embedded. 
The trace of a Birkhoff annulus is the image of a lift under the projection $\widetilde{M} \to \mathcal{O}$.
Since $L'$ is positive, the boundary components of $A$ are positive.
Performing Fried resolution (see \cite{Fri83}) on $A$ gives a positive partial section $T$ to $\phi$.
Thus by \Cref{prop:+BASimpliesnonegativepartialsection}, $\phi$ cannot be $-$BAS.

Similarly, if $\phi$ has a negative lozenge that is disjoint from $\sing(\phi)$, then $\phi$ cannot be $+$BAS.
We conclude that if $\phi$ has both positive and negative lozenges that are disjoint from $\sing(\phi)$, then $\phi$ cannot be BAS.

Examples of pseudo-Anosov flows that satisfy this condition include all non-$\mathbb{R}$-covered Anosov flows. 
Here, recall that each Anosov flow is either trivial, skew, or non-$\mathbb{R}$-covered. 
See \cite{BI23} for methods of constructing non-$\mathbb{R}$-covered Anosov flows.
\end{noneg}

We conjecture that the converse of \Cref{noneg:positiveandnegativelozenges} is also true.

\begin{conj} \label{conj:lozenges}
Let $\phi$ be a transitive pseudo-Anosov flow on an oriented closed 3-manifold. If $\phi$ does not admit both positive and negative lozenges disjoint from $\sing(\phi)$, then $\phi$ is BAS.
\end{conj}

\begin{rmk}
    Examples of transitive pseudo-Anosov flows that do not admit both positive and negative lozenges disjoint from $\sing(\phi)$, and are not yet known to be BAS, can be constructed for instance in the following two ways:
    \begin{itemize}
        \item Following the construction developed in \cite[section 8]{BF13}, one may build \emph{totally periodic} pseudo-Anosov flows for which all lozenges have corners on $\sing(\phi)$. To do that, it suffices to build the examples from admissible fat graphs all of whose vertices have valence $>4$. Note that this family of examples do admit both positive and negative lozenges. See \cite{BF15} or \cite[Section 5.5]{BM25} for more details.
        \item Taking branched covers of geodesic flows (or more generally contact Anosov flows) as in \cite[Remark 5.10]{BF21} or \cite[Example 2.12]{ChaPan25} yields a different family of examples. Up to changing the orientation of the manifold, all lozenges of flows in this family are positive.
    \end{itemize}

Note that the finiteness conjecture is already known for both of these families, by \cite{BF15} for the totally periodic case, and by \cite{ChaPan25} for the branched covers of contact Anosov flows.
\end{rmk}

It would also be interesting to have a better understanding on the relationship between pseudo-Anosov flows with positive sections, pseudo-Anosov Reeb flows in the sense of \cite{ChaPan25} and BAS pseudo-Anosov flows. There exist flows admitting both a positive Birkhoff section and a transverse surface; such a flow is BAS, but not pseudo-Anosov Reeb.
Finally, there exist examples of BAS pseudo-Anosov flows that are neither pseudo-Anosov Reeb nor admit a strictly positive Birkhoff section. For instance, one can construct many pseudo-Anosov flows without perfect fits that admit a transverse surface (e.g., a suspension of a pseudo-Anosov diffeomorphism), and therefore cannot be pseudo-Anosov Reeb (\cite[Lemma 2.13]{ChaPan25}). Similarly suspensions of pseudo-Anosov diffeomorphisms cannot have a strictly positive Birkhoff section.

\subsection{Other rigid orbits}

Morally, \Cref{thm:finitenessBAS} builds on \cite{BBowM24} by using the fact that singular orbits of a pseudo-Anosov flow are \textbf{rigid}, in the sense that on a given 3-manifold, there are at most finitely many possibilities for the isotopy class of the singular orbits and their degeneracy curves. A natural question is then: Are there other rigid types of orbits for a pseudo-Anosov flow?
More specifically, we make the following conjecture.

\begin{conj} \label{conj:branchingorbitrigid}
A \textbf{branching orbit} of a pseudo-Anosov flow $\phi$ is a closed orbit that lies on a leaf of $\Lambda^s$ or $\Lambda^u$ that is part of a cataclysm in the leaf space.
The set of branching orbits of a pseudo-Anosov flow is rigid, i.e. on a given 3-manifold, there are at most finitely many possibilities for the isotopy class of the branching orbits and their degeneracy curves.
\end{conj}

If \Cref{conj:branchingorbitrigid} is true, then Theorem \ref{thm:relposfinite} would give finiteness of pseudo-Anosov flows that admit a Birkhoff section all of whose negative boundary components lie along branching and singular orbits.

\subsection{Effective bounds}

It might be interesting to have effective bounds on the number of (orbit equivalence classes of) pseudo-Anosov flows on a given 3-manifold $M$.
As mentioned in the introduction, our finiteness result comes from finiteness of tight contact structures \cite{CGH09} and finiteness of isotopy classes of singular orbits \cite{Li26}.
Since the proof of both results involve some sort of `normal form' with respect to some triangulation $\Delta$ on $M$, one might expect to get a bound in terms on the number of tetrahedra in $\Delta$.
More specifically, we ask:

\begin{quest}
Can one give an explicit exponential bound on the number of (orbit equivalence classes of) pseudo-Anosov flows without perfect fit on a closed 3-manifold $M$ in terms of the number of tetrahedra in a triangulation of $M$?
\end{quest}

\bibliographystyle{alphaurl}

\bibliography{bib.bib}

\end{document}